\documentclass[a4paper,11pt]{article}

\usepackage{amsmath}
\usepackage{amssymb}
\usepackage{amsthm}
\usepackage{amsfonts}
\usepackage{mathtools}
\usepackage{graphicx}
\usepackage{tikz,pgfplots}
\usepackage[utf8]{inputenc}
\usepackage{pgfplots}
\usepackage{tikz}
\pgfplotsset{compat=1.9}
\usepackage{subcaption}
\usepackage{caption}

\usepackage{titlesec}
\usepackage{enumitem}
\usepackage{mathtools}
\usepackage{listings}
\usepackage{geometry}
\usepackage{bbm}
\usepackage{xcolor}
\usepackage{hyperref}
\usepackage{caption}
\usepackage{mathrsfs}
\usepackage{booktabs}
\usepackage{floatrow}
\newfloatcommand{capbtabbox}{table}[][\FBwidth]
\usepackage{blindtext}
\usepackage{colortbl}
\newtheorem{theorem}{Theorem}[section]
\newtheorem{corollary}[theorem]{Corollary}
\newtheorem{proposition}[theorem]{Proposition}
\newtheorem{lemma}[theorem]{Lemma}
\newtheorem{remark}[theorem]{Remark}
\newtheorem{definition}[theorem]{Definition}

\newcommand{\phase}{\omega}

\newcommand{\re}{\mbox{\rm Re}}
\newcommand{\im}{\mbox{\rm Im}}

\makeatletter
\newcommand*{\rom}[1]{\expandafter\@slowromancap\romannumeral #1@}
\makeatother

\newcommand{\0}{L^2(\D)}
\newcommand{\1}{H^1(\D)}

\newcommand{\3}{H^{1}_{0}(\D)}
\newcommand{\5}{\tangentspace{u} \cap \tangentspace{\ci u}}%{\mathscr{H}_{u}^{\mathbb{C} , \perp }}

\newcommand{\narrowminus}{\hspace{-2pt}-\hspace{-1pt}}

\newcommand{\eps}{\varepsilon}

\newcommand{\D}{\mathcal{D}}

\newcommand{\lm}{\lambda}

\newcommand{\Zu}{\mathcal{Z}}

\newcommand{\R}{\mathbb{R}}
\newcommand{\C}{\mathbb{C}}

\newcommand{\Ltwo}[2]{(#1,#2)_{L^2(\D)}}
\newcommand{\Hone}[2]{(#1,#2)_{H^1(\D)}}

\newcommand{\Acalinv}[1]{\mathcal{A}^{-1}_{|#1|} \mathcal{I}}

\definecolor{mycolor1}{rgb}{0.00000,0.44700,0.74100}%

\newcommand{\quotes}[1]{``#1''}

\newcommand{\dx}{\hspace{2pt}\mbox{d}x}

\newcommand{\ci}{\mathrm{i}} %% complex number "i"

\newcommand{\sR}{\mbox{\rm \tiny R}}

\newcommand{\nablaR}{\nabla_{\hspace{-2pt}\sR}}
\newcommand{\VR}{V_{\hspace{-1pt}\sR}}

\newcommand{\tangentspace}[1]{T_{#1}\mathbb{S}}

\newcommand{\Acal}{\mathcal{A}}

\begin{document}
\begin{center}
{\LARGE 
Convergence of a Riemannian gradient method for the Gross--Pitaevskii energy functional in a rotating frame\renewcommand{\thefootnote}{\fnsymbol{footnote}}\setcounter{footnote}{0}
 \hspace{-3pt}\footnote{The authors acknowledge the support by the German Research Foundation (DFG grant HE 2464/7-1).}}\\[2em]
\end{center}

\begin{center}
{\large Patrick Henning\footnote[1]{\label{affiliation}Department of Mathematics, Ruhr-University Bochum, DE-44801 Bochum, Germany.\\ email: \href{mailto:patrick.henning@rub.de}{patrick.henning@rub.de} and \href{mailto:mahima.yadav@rub.de}{mahima.yadav@rub.de} }
and Mahima Yadav\textsuperscript{\ref{affiliation}}
}\\[2em]
\end{center}

\begin{center}
{\large March 18, 2025} %{\large{\today}}
\end{center}

\begin{abstract}
This paper investigates the numerical approximation of ground states of rotating Bose-Einstein condensates. This problem requires the minimization of the Gross--Pitaevskii energy $E$ on a Hilbert manifold $\mathbb{S}$. To find a corresponding minimizer $u$, we use a generalized Riemannian gradient method that is based on the concept of Sobolev gradients in combination with an adaptively changing metric on the manifold. By a suitable choice of the metric, global energy dissipation for the arising gradient method can be proved. The energy dissipation property in turn implies global convergence to the density $|u|^2$ of a critical point $u$ of $E$ on $\mathbb{S}$. Furthermore, we present a precise characterization of the local convergence rates in a neighborhood of each ground state $u$ and how these rates depend on the first spectral gap of $E^{\prime\prime}(u)$ restricted to the $L^2$-orthogonal complement of $u$. With this we establish the first convergence results for a Riemannian gradient method to minimize the Gross--Pitaevskii energy functional in a rotating frame. At the same time, we refine previous results obtained in the case without rotation. The major complication in our new analysis is the missing isolation of minimizers, which are at most unique up to complex phase shifts. For that, we introduce an auxiliary iteration in the tangent space $T_{\mathrm{i} u} \mathbb{S}$ and apply the Ostrowski theorem to characterize the asymptotic convergence rates through a weighted eigenvalue problem. Afterwards, we link the auxiliary iteration to the original Riemannian gradient method and bound the spectrum of the weighted eigenvalue problem to obtain quantitative convergence rates. 
Our findings are validated in numerical experiments.
\end{abstract}

\section{Introduction}
At extreme temperatures approaching zero Kelvin, dilute bosonic gases can undergo a fascinating transition to a state of matter known as Bose--Einstein condensate (BEC), cf. \cite{Bos24,Ein24,PiS03}. One of the properties that sets it apart from other states of matter is its extraordinary superfluid nature. Experimentally, the superfluidity of a BEC can be confirmed by letting it rotate by means of a stirring potential. If the rotation is sufficiently strong, quantized vortices appear as a distinctive sign of superfluidity \cite{MAH99}. Furthermore, the larger the angular velocity (relative to the strength of the trapping potential), the more vortices appear in the condensate. In this paper, we are concerned with the numerical approximation of ground states of such a rotating BEC. Ground states are the stable stationary states at the lowest possible energy level, which are mathematically defined as (constrained) global minimizers of the Gross--Pitaevskii energy functional. As derived in the comprehensive review papers by Bao et al. \cite{Bao14,BaC13b}, the Gross--Pitaevskii energy functional in a rotational frame is given by
\begin{align*}
E(v):= \frac{1}{2}\int_{\mathcal{D}} |\nabla v|^2 + V\, |v|^2  - \Omega\, \bar{v}\, \mathcal{L}_{3}v + \frac{\beta}{2} |v|^4 \dx
\end{align*}
for $v \in H^1_0(\D;\C)$.  Here, $\D \subset \R^d$ (for $d=2,3$) is the computational domain, $V \in L^{\infty}(\D;\R_{\ge 0})$ represents a trapping potential, $\beta \in \R_{\ge }0$ describes the strength of (repulsive) particle interactions, $\Omega \in \R$ is the angular velocity of a stirring potential and $\mathcal{L}_3 = - \ci \left( x_1 \partial_{x_2} - x_2 \partial_{x_1} \right)$ denotes the $x_3$-component of the angular momentum (which formally involves a dimension reduction for $d=2$, cf. \cite{BaC13b}). A ground state of a rotating BEC is now defined as global minimizer of the energy $E$ on the $L^2$-unit sphere 
\begin{align}
\label{def-L2-spehere} \mathbb{S} := \{ v \in H^1_0(\D;\C) \,| \, \|v\|_{L^2(\D)} = 1 \}.
\end{align}
The density of a ground state $u$ is associated with the real-valued (and physically observable) quantity $|u|^2$. Consequently, the normalization condition $u \in \mathbb{S}$, or equivalently $\int_{\D} |u|^2 \dx =1$, can be seen a constraint for the mass of the condensate.

Numerical schemes for computing a ground state consist of two crucial components: the choice of an iterative solver to approach a (constrained) energy minimizer and the choice of a spatial discretization (i.e., a discrete space in which $E$ is minimized). Regarding the latter aspect, common space discretizations involve spectral and pseudo-spectral methods \cite{AnD14,Bao14,BaC13b,CCM10}, Lagrange finite element methods (FEM) \cite{CCM10,CGHZ11,DaH10,HSW21,PHMY24}, discontinuous Galerkin composite finite elements \cite{EngGianGrub22}, spectral element methods \cite{CLLZ24-discrete} and generalized finite elements \cite{HMP14b,HeP23}. In particular, optimal order error estimates for (conforming) FEM approximations of ground states of rotating BECs were first derived in \cite{PHMY24}. 

In this paper, we will discard the spatial discretization and instead focus on the second crucial aspect of a suitable numerical scheme: the construction of iterative solvers for finding minimizers of $E$ and the corresponding error analysis. The majority of suitable iterative solvers in the literature is obtained from the perspective of (Riemannian) gradient flows \cite{ALT17,AltPetSty22,BaD04,BWM05,CDLX23,CLLZ24,DaK10,DaP17,KaE10,Zhang2022}. However, the problem can be also regarded from the perspective of eigenvalue solvers by considering the Euler--Lagrange equations of the constrained minimization problem. These equations, known as the Gross--Pitaevskii eigenvalue problem, can be approached with (nonlinear) eigenvalue solvers such as generalized inverse iterations \cite{AHP21NumMath,JarKM14,PH24} or self consistent field iterations \cite{Can00,CaL00,CaL00B,DiC07} with corresponding convergence results obtained in \cite{CKL21}. Finally, the problem can be also tackled with various Newton-type methods \cite{APS23Newton,WWB17,XXXY21}. Despite the different perspectives, many of these iterative methods are closely related or even equivalent and we refer to \cite{HenJar24} for a recent review on the topic.

Although there is a large variety of different methods, analytical convergence results are rare in the literature and were only obtained in recent years. To the best of our knowledge, there is only one contribution that addresses convergence of an iterative solver for rotating BECs, that is, for the $J$-method \cite{AHP21NumMath} (a particular inverse iteration originally proposed by Jarlebring et al. \cite{JarKM14}). In the paper at hand we consider a different class of solvers, so-called discrete Sobolev gradient methods \cite{DaK10,DaP17,KaE10} (which are generalized Riemannian gradient methods). So far, these methods have been only analyzed in settings without rotation (i.e. $\Omega =0$), where the first results were obtained by Faou and J\'{e}z\'{e}quel \cite{FaT18} who proved local convergence for a discrete $L^2$-gradient flow (DNGF) in space dimension $d=1$ and for interaction constants $\beta \le 0$. Later, these findings were generalized in \cite{PH24} to $d=1,2,3$ and $\beta \ge 0$. For discrete $H^1$-gradient flows, convergence was recently established in \cite{CLLZ24-discrete,CLLZ24}, again for $d=1,2,3$ and $\beta\ge0$. Finally, in the same setting, local convergence for energy-adaptive gradient flows was proved in \cite{CLLZ24,PH24,Zhang2022} and global convergence in \cite{HLP24,HeP20}. Here we note that \cite{CLLZ24-discrete} and \cite{HLP24} are currently the only works that consider iterations in a fully discrete setup.

The main goal of this paper is to generalize the aforementioned energy-adaptive gradient flow to the case of rotating BECs and to extend the convergence results of \cite{CLLZ24,PH24,HeP20,Zhang2022} accordingly. The construction of the particular gradient flow is based on selecting an adaptively changing metric, encoded in a Hilbert space $X$, such that the corresponding Sobolev gradient at a fixed location $v\in \mathbb{S}$ fulfills $\nabla_X E(v)=v$. With this choice, the corresponding gradient flow is purely driven by the $X$-orthogonal projection in the tangent space. The idea was first proposed in \cite{HeP20} to achieve global energy dissipation (in the case without rotation) and the approach was later transferred to the Stiefel manifold \cite{AltPetSty22}. By discretizing the arising gradient flow, a generalized Riemannian gradient method is obtained. As mentioned above, linear order convergence for $\Omega=0$ was proved in \cite{CLLZ24,Zhang2022} using energy-descent ({\L}ojasiewicz-type) inequalities and in \cite{PH24} through weighted eigenvalue problems. However, all three works crucially exploit the following property: if $u$ is a ground state with ground state eigenvalue $\lambda>0$, i.e., $E^{\prime}(u)= \lambda \,(u,\cdot)_{L^2(\D)}$, then $\lambda$ is also the simple smallest eigenvalue of the linearized Gross--Pitaevskii operator $\mathcal{A}_{|u|}v := - \Delta v + V\, v + \beta |u|^2 v$ (cf. \cite{CCM10} for a proof of the result). Unfortunately, this result is no longer true for rotating BECs (cf. Section \ref{section-numerical-exp}), so that the proof techniques are not directly applicable anymore. Besides, the minimization problem with rotation yields additional analytical challenges such as the step from a fully real-valued setting to a complex-valued setting and, related to this, missing uniqueness of ground states \cite{BWM05}. In fact, the energy is invariant under complex phase shifts of the form $u \mapsto \exp(\ci \phase) u$ for any $\phase\in [-\pi, \pi)$, which implies that any such phase shift of a ground state $u$ remains a ground state, though both share the same density $|u|^2=|\exp(\ci \phase) u|^2$ (which is also the physically relevant quantity). Since the ground states are not isolated (due to the invariance under phase shifts), we do not even have local uniqueness.

Our analysis of the energy-adaptive discrete gradient flow, which we call from now on {\it energy-adaptive Riemannian gradient method}, is based on first establishing strict energy dissipation for the iterates and global convergence of the iterated densities to, at least, a critical point of the Gross--Pitaevskii energy on $\mathbb{S}$. Locally in a neighborhood of a ground state we further establish a quantitative result on the local linear convergence of the densities. For that, we introduce an auxiliary iteration which has the same structure as the original scheme, but which artificially adjusts the phase in each iteration so that the iterates are in the same phase as a particular ground state. The local convergence of the auxiliary iterates can now be approached with abstract theorems on the convergence of fixed-point iterations (Ostrowski theorem), which in turn allows us to characterize the rates through a weighted eigenvalue problem in the tangent space $\tangentspace{u}$ of the ground state (which defines the reference phase). By studying the spectrum of this eigenvalue problem, we can give an accurate characterization of the convergence properties. In particular, we can prove how the rates depend on the first spectral gap of $E^{\prime\prime}(u) \vert_{\tangentspace{u}}$. Finally, the convergence results are transferred to the original gradient method. Even in the case without rotation, our findings extend and sharpen the results of \cite{CLLZ24,PH24,HeP20,Zhang2022} for the energy-adaptive Riemannian gradient method.\\[0.4em]
{\it Overview.} The paper is organized as follows: Section \ref{section-a_u-scheme} establishes preliminary notations and introduces the minimization problem along with corresponding well-posedness results. In Section \ref{section-DNSGF} we motivate and define 
the {\it energy-adaptive Riemannian gradient method} and formulate our main results regarding global and local convergence of the proposed scheme. 
The corresponding proof of global convergence together with its energy diminishing property are given in Section \ref{section-global-convergence}. Local convergence and the characterization of rates is proved in Section \ref{section-local-convergence-proof}. Finally, Section \ref{section-numerical-exp} is devoted to numerical tests to validate the theoretical findings.
\section{Mathematical setting}
\label{section-a_u-scheme}
In this section, we introduce the basic notation, our set of assumptions and the precise problem setting including well-posedness. In what follows, we always assume that
\begin{enumerate}[label={(A\arabic*)}]
\item \label{A1} $\D \subset \mathbb{R}^d$ is a bounded convex domain for $d=2,3$.
\item\label{A2} $V \in L^{\infty}(\mathcal{D},\mathbb{R}_{\ge 0} )$ denotes a non-negative potential and
\item\label{A3} the parameter $\beta \in \R_{\ge 0}$ characterizes the strength of (repulsive) particle interactions.
\end{enumerate}
Since we consider the minimization of a real-valued functional $E$, we need to work with {\it real-linear} spaces consisting of complex-valued functions (cf. \cite{AHP21NumMath,Begout22}). For that reason, we equip the Lebesgue space $L^2(\mathcal{D}):= L^2(\mathcal{D};\C)$ and the Sobolev space $H_{0}^{1}(\mathcal{D}):= H_{0}^{1}(\mathcal{D};\mathbb{C})$ 
respectively with the following real inner products 
\begin{align*}
\Ltwo{v}{w}:= \re \Big( \int_{\D} v \, \overline{w} \dx \Big) \quad \mbox{and} 
\quad \Hone{v}{w}:= \re \Big(  \int_{\mathcal{D}} v \, \overline{w} \dx + \int_{\mathcal{D}} \nabla v \cdot \overline{\nabla w} \dx \Big).
\end{align*}
Here, $\overline{w}$ denotes the complex conjugate of $w$. The corresponding $(\text{real})$ dual space is denoted by $H^{-1}(\mathcal{D}) := \big(H_{0}^{1}(\D)\big)^{*}$, with canonical duality pairing $ \langle \cdot , \cdot \rangle := \langle \cdot ,\cdot \rangle _{H^{-1}(\mathcal{D}) , H^{1}_{0}(\mathcal{D})}$.

\subsection{Problem formulation }
\label{subsection-problem-formulation}
We begin by introducing the  Gross--Pitaevskii energy functional $E :H_{0}^{1}(\D) \rightarrow \mathbb{R}$ with %for $v\in H^1_0(\D)$ by
\begin{align}
\label{energy1}
E(v)
:= \frac{1}{2}\int_{\mathcal{D}} |\nabla v|^2 + V\, |v|^2  - \Omega\, \bar{v}\, \mathcal{L}_{3}v + \frac{\beta}{2} |v|^4 \dx
\qquad
\mbox{for } v\in H^1_0(\D).
\end{align}
Here, $\mathcal{L}_3 := - \ci \left( x_1 \partial_{x_2} - x_2 \partial_{x_1} \right)$ denotes the third component of the angular momentum operator, hence describing a rotation around the $x_3$-axis of the coordinate system. The corresponding angular velocity of this component is given by $\Omega \in \mathbb{R}$. Recalling the $L^2$-unit sphere $\mathbb{S}$ from \eqref{def-L2-spehere}, we are concerned with the following Riemannian minimization problem:
\begin{align}
\label{minimization-problem}
\mbox{Find } u \in \mathbb{S}: 
\hspace{20pt}
E(u) = \underset{v \in \mathbb{S}}{\inf}\hspace{2pt} E(v).
\end{align}
Any such minimizer $u$ is called a {\it ground state} of the rotating Bose--Einstein condensate. Physically speaking, ground states are the stable stationary states at the lowest possible energy level (under a constraint for the mass of the BEC, represented by the condition $u  \in \mathbb{S}$). Since every ground state $u$ is a critical point of $E$ on $\mathbb{S}$, we can characterize it by the corresponding Euler--Lagrange equation seeking $u\in \mathbb{S}$ with Lagrange multiplier $\lambda \in \R$ such that
\begin{align}
\label{eigenvalue-E-derivative-form}
\langle E^{\prime}(u) , v \rangle = \lambda \, ( u , v )_{L^2(\D)}  \qquad \mbox{for all } v\in H^1_0(\D),
\end{align}
where $E^{\prime}: H^1_0(\D) \rightarrow H^{-1}(\D)$ denotes the Fr\'echet derivative of $E$ and the right hand side is the derivative of the constraint functional $\tfrac{1}{2}( \| v \|_{L^2(\D)}^2 - 1)$. Since we can interpret $\lambda$ as an eigenvalue with eigenfunction $u$, problem \eqref{eigenvalue-E-derivative-form} is known as the {\it Gross--Pitaevskii eigenvalue problem} (GPEVP). The derivative $E^{\prime}$ is straightforwardly computed as
\begin{align}
\label{Eprime-def}
\langle E^{\prime}(u) , v \rangle
:= ( \nabla u , \nabla v)_{L^2(\D)} + ( V\, u , v )_{L^2(\D)}  - \Omega\, ( \mathcal{L}_{3} u , v )_{L^2(\D)} + \beta ( |u|^2 u , v )_{L^2(\D)}.
\end{align}
Consequently, the GPEVP \eqref{eigenvalue-E-derivative-form} can be written as: find $\lambda > 0$ and $u \in\mathbb{S}$ such that
\begin{align}
\label{eigen_value_problem_1}
-\Delta u + V\, u - \Omega \, \mathcal{L}_{3}u + \beta \, |u|^2u = \lambda \, u.
\end{align}
Any eigenvalue $\lambda$ that belongs to a ground state $u$ is called a {\it ground state eigenvalue}. Note that a ground state eigenvalue is not necessarily the smallest eigenvalue of \eqref{eigen_value_problem_1}, cf. the numerical experiments in \cite[Section 6.3]{AHP21NumMath} and \cite[Section 4.3]{AHYY24}. 
\subsubsection{Well-posedness: existence and local quasi-uniqueness}
For a proper analysis, we first need to discuss the existence and (missing) uniqueness of ground states to the minimization problem \eqref{minimization-problem}. In order for a ground state to exist, the strength of the trapping potential needs to be sufficiently strong compared to the angular velocity $\Omega$. Following \cite{BWM05}, we assume that:
\begin{enumerate}[resume,label={(A\arabic*)}]
\item\label{A4} There is a constant $K > 0$ such that
\begin{align*}
	V(x) - \frac{1 + K}{4} \Omega^2 (x_1^2 + x_2^2) \ge  0 \quad \text{for almost all } x \in \D.
\end{align*}
\end{enumerate}
Physically speaking, the assumption states that the trapping frequencies of the potential $V$ in $x_1$- and $x_2$-direction are larger than the frequency $\Omega$ of the stirring potential (which triggers the rotation around the $x_3$-axis). A violation of this condition leads (typically) to the non-existence of minimizers, cf. \cite{BWM05}. On the contrary, if \ref{A4} holds, existence of ground states can be proved with standard techniques, where we refer exemplarily to \cite[Section 2]{PHMY24} for more details. In particular, the following holds. 
\begin{proposition}[Existence of ground states]
Assume \ref{A1}-\ref{A4}, then there exists at least one ground state $u \in \mathbb{S}$ to problem \eqref{minimization-problem} and it holds $E(u)>0$.
\end{proposition}
As mentioned in the introduction, ground states can be at most unique up to phase shifts since $E(u) = E(\exp(\ci \phase) u)$ for any $\phase \in [-\pi , \pi)$. We call such ground states that only differ in the phase factor $\exp(\ci \phase)$ (gauge) {\it equivalent}. Since equivalent ground states share the same density, i.e.,  $|u|^2=|\exp(\ci \phase) u|^2$, one might wonder if there is hope for a unique ground state density $|u|^2$. Empirically, this seems to be commonly the case, but exceptions are possible, as first pointed out in \cite[Section 6.2]{BWM05}. For a detailed discussion of missing uniqueness we refer to \cite{PHMY24}.

Despite the general lack of uniqueness, it is reasonable to assume that ground states are at least locally quasi-unique, i.e., in any sufficiently small neighborhood of a ground state $u$, the only other ground states are phase shifts of $u$. To make this statement precise, let
$$
\tangentspace{u} := \{ \, v \in H^1_0(\D) \, | \, (u,v)_{\0} = 0 \, \}
$$
denote the tangent space at $u$ and recall that if $\gamma : (-1,1) \rightarrow \mathbb{S}$ is a curve on $\mathbb{S}$ with, e.g., $\gamma(0)=u \in \mathbb{S}$, then it must naturally hold $\gamma^{\prime}(0) \in \tangentspace{u}$ because the slope of the curve is tangential to the sphere in any point. Furthermore, any curve that follows the direction of the phase shifts, e.g. locally $\gamma(t)= \exp( \alpha \, \ci \, t )u$ for fixed $\alpha \in \R \setminus \{0\}$, has the slope $\gamma^{\prime}(0)= \alpha \, \ci u$. Hence, if $\gamma$ is a curve on $\mathbb{S}$ with $\gamma(0)=u$, then the direction of phase shifts of $u$ is blocked if $\gamma^{\prime}(0) \in (\tangentspace{u} \cap \tangentspace{\ci u})$. In this case, $t \mapsto E(\,\gamma(t)\,)$ has a strict local minimum in $t=0$ if $u$ is a locally unique ground state (up to phase shifts).
With the above considerations we follow the definition in \cite{PHMY24} and introduce the notion of {\it locally quasi-unique ground states} below.
\begin{definition}[Locally quasi-unique ground state] 
\label{definition-quasi-isolated-ground-state}
A ground state $u$ to \eqref{minimization-problem} is called \emph{locally quasi-unique} if for any smooth curve $\gamma : (-1,1) \rightarrow \mathbb{S}$  with $\gamma(0)=u$ and $\gamma^{\prime}(0) \in (\tangentspace{u} \cap \tangentspace{\ci u}) \setminus \{ 0 \}$ it holds
 $$ \frac{ \mbox{d}^{\,2} }{ \mbox{d}\hspace{1pt}t^{2}} E(\gamma(t)) |_{t=0} > 0.$$
\end{definition} 
The definition precisely guarantees that $u$ is an isolated minimum of $E$ up to phase shifts.
The property can be checked a posteriori by numerically computing the spectrum of $E^{\prime\prime}(u) \vert_{\tangentspace{u}}$ once $u$ is available (cf. the subsection below, as well as Section \ref{section-numerical-exp}). In fact, in numerous experiments over the last years, we never encountered a ground state that was not locally quasi-unique in the sense of Definition \ref{definition-quasi-isolated-ground-state}. Hence, assuming local quasi-uniqueness of a ground state is not restrictive and, in our experience, covers (at least) the vast majority of ground states encountered in practice (if not all of them).

\subsection{The spectrum of $E^{\prime\prime}(u) \vert_{\tangentspace{u}}$}
In this section, we will introduce some simplifying notations and we will have a closer look at the spectrum of $E^{\prime\prime}(u)$. Recalling the Gross--Pitaevskii operator $E^{\prime}(u)$ from \eqref{Eprime-def}, we introduce its linearized version $\Acal_{|u|} : H^1_0(\D) \rightarrow H^{-1}(\D)$ for any $u\in H^1_0(\D)$ by
\begin{align}
\label{defAcal-u-square}
\langle \Acal_{|u|} v , w \rangle
:= ( \nabla v , \nabla w)_{L^2(\D)} + ( V\, v , w )_{L^2(\D)}  - \Omega\, ( \mathcal{L}_{3} v , w )_{L^2(\D)} + \beta ( |u|^2 v , w )_{L^2(\D)}.
\end{align}
In particular, we have $E^{\prime}(u)= \Acal_{|u|}u$. In the case without rotation (i.e. $\Omega=0$), it can be shown that, for any ground state $u$, the smallest eigenvalue of the $u$-linearized operator $\Acal_{|u|}$ is identical to the ground state eigenvalue $\lambda$ (cf. \cite{CCM10}). In our setting with rotation, this is typically no longer the case, as demonstrated numerically in Section \ref{section-numerical-exp}. Consequently, we cannot exploit this property (used in previous works \cite{CLLZ24,PH24,Zhang2022}) for our convergence analysis and we will work with the spectrum of $E^{\prime\prime}(u) \vert_{\tangentspace{u}}$ instead.

Before we turn to $E^{\prime\prime}(u)$, note that the $u$-independent part of $\Acal_{|u|}$ can be expressed by a symmetric and coercive bilinear form. For that, we define $\VR(x):= V(x) - \tfrac{1}{4}\, \Omega^2\, |x|^2$ (which is non-negative by \ref{A4}) and the covariant gradient $\nablaR$ by
\begin{align}
\label{def-nablaR}
\nablaR v := \nabla v + \ci \tfrac{\Omega}{2} \text{R}^{\hspace{-1pt}\top} v,
\qquad
\mbox{where } 
\text{R}(x) := \begin{cases}
(x_2,-x_1) & \mbox{for } d=2 \\
(x_2,-x_1,0) & \mbox{for } d=3
\end{cases}.
\end{align}
With this, we define the symmetric bilinear form
\begin{align}
\label{R-inner-product-clone}
(v,w)_{\sR} := \re \Big(  \int_{\mathcal{D}} \nablaR v \cdot \overline{\nablaR w} \dx + \int_{\mathcal{D}} \VR v \, \overline{w} \dx  \Big).
\end{align}
Under our assumptions, $(\cdot,\cdot)_{\sR}$ is an inner product on $H^1_0(\D)$ and its induced norm $\| v \|_{\sR}:= \sqrt{(v ,v)_{\sR}}$ is equivalent to the standard $H^1$-norm, cf. \cite{DaK10,pc22} for a proof. Using the above notation, the energy and the operator $ \Acal_{|u|}$ can be compactly written as
\begin{align}
\label{energy_function-clone}
E(v) = \tfrac{1}{2} \| v \|_{\sR}^2 +  \tfrac{\beta}{4} \int_{\D} |v|^4 \dx
\qquad
\mbox{and}
\qquad
\langle \Acal_{|u|}v , w \rangle = (v,w)_{\sR} + \beta \, ( |u|^2 v,w )_{\0}.
\end{align} 
The above expression shows that $\Acal_{|u|}$ is $L^2$-self-adjoint and elliptic (using the equivalence of the norms $\| v \|_{\sR}$ and $\| v\|_{H^1(\D)}$). Furthermore, we can express the Gross--Pitaevskii eigenvalue problem  \eqref{eigenvalue-E-derivative-form} compactly as: find $u\in \mathbb{S}$ and $\lambda >0$ such that
\begin{align}
 \label{eigen_value_problem_2}
\langle \Acal_{|u|}u , v \rangle = \lambda \, ( u , v )_{L^2(\D)} \qquad \mbox{for all } v\in H^1_0(\D).
\end{align} 
With the above notation, we can now turn towards $E^{\prime\prime}(u)$ which is computed for $u\in \mathbb{S}$ as 
\begin{align} 
\label{sd}
\langle E''(u)v,w \rangle =  \langle \Acal_{|u|}v,w \rangle +  2\, \beta \, ( \re (u \overline{v}) \, u,w)_{\0} \quad \mbox{for } v,w \in \3. 
\end{align}
As discussed in the previous subsection, for any ground state $u\in \mathbb{S}$, the smallest eigenvalue of $E^{\prime\prime}(u)\vert_{\tangentspace{u}}$ must coincide with the corresponding ground state eigenvalue $\lambda$ given by \eqref{eigen_value_problem_2}. Indeed, as $\ci u \in \tangentspace{u}$ we easily calculate that
\begin{align*} 
\langle E''(u)\ci u, v \rangle =  \langle \Acal_{|u|}\ci u, v \rangle +  2\, \beta \, ( \re (u \overline{\ci u}) \, u,v )_{\0} 
=    \langle \Acal_{|u|}\ci u, v \rangle  =  \lambda \, ( \ci u , v )_{L^2(\D)}
\end{align*}
for any $v\in H^1_0(\D)$. The eigenfunction $\ci u$ to the eigenvalue $\lambda$ just corresponds to the direction of complex phase shifts. The following lemma states that $\lambda$ is a simple eigenvalue of $E''(u)\vert_{\tangentspace{u}}$ if $u$ is a locally quasi-unique ground state.
\begin{lemma} \label{coercive_lemma} Assume \ref{A1}-\ref{A4} and let $u$ be a ground state to \eqref{minimization-problem} that is locally quasi-unique in the sense of Definition \ref{definition-quasi-isolated-ground-state} and with corresponding ground state eigenvalue $\lambda$. Furthermore, let $0 < \lambda_1 \le \lambda_2 \le \cdots $ denote the ordered eigenvalues of $E''(u)\vert_{\tangentspace{u}}$, i.e., $v_i \in \tangentspace{u}$ and $\lambda_i \in \R$ solve the eigenvalue problem   
\begin{align}
\label{spectrum-E''u}
\langle E''(u)  v_i  , w \rangle = \lambda_i \hspace{1pt}(v_i , w )_{L^2(\D)} \qquad
\mbox{for all } w \in \tangentspace{u}.
\end{align}
Then $\lambda_1>0$ is a simple eigenvalue to the eigenfunction $v_1 = \ci u$ and it holds $\lambda_1 = \lambda$. Furthermore $\langle E''(u) \, \cdot , \cdot \rangle - \lambda (\cdot ,\cdot)_{L^2(\D)}$ is coercive on $\tangentspace{u} \cap \tangentspace{\ci u}$ with 
\begin{align}
\label{coercivity-secE-lambda}
\langle E''(u) \, v , v \rangle - \lambda \, (v ,v)_{L^2(\D)}
&\ge 
  \tfrac{1}{2} \min \{ 1, \tfrac{\lm_2}{\lm_1}-1\} \|v\|^2_{\sR}
 \qquad \mbox{for all } v \in \tangentspace{u} \cap \tangentspace{\ci u}.
\end{align}
In particular, by $\lambda_1<\lambda_2$ and the norm equivalence of $\| \cdot \|_{\sR}$ and $\| \cdot \|_{H^1(\D)}$ there is a constant $C>0$ such that it holds $\langle E''(u) \, v , v \rangle - \lambda (v ,v)_{L^2(\D)} \ge C  \|v\|^2_{H^1(\D)}$ for all $ v \in \tangentspace{u} \cap \tangentspace{\ci u}$.
\end{lemma}
\begin{proof}
We already verified that $\lambda$ is an eigenvalue of $E''(u)\vert_{\tangentspace{u}}$ with eigenfunction $\ci u$. Since $u$ is a ground state (global minimizer) we have $0 \le \tfrac{\mbox{\scriptsize d}^2}{\mbox{\scriptsize d}t^2}E(\gamma(t))=\langle E''(u) \, v , v \rangle - \lambda \, (v ,v)_{L^2(\D)}$ for any smooth curve $\gamma:(-1,1) \rightarrow \mathbb{S}$ with $\gamma(0)=u$ and $\gamma^{\prime}(0)=v \in \tangentspace{u}$. Hence, there cannot be a smaller eigenvalue than $\lambda$ and therefore $\lambda=\lambda_1$. In \cite[Proposition 2.4]{PHMY24} it was proved that $\langle E''(u) \, \cdot , \cdot \rangle - \lambda \, ( \cdot ,\cdot )_{L^2(\D)}$ is inf-sup stable (and consequently non-singular) on $\tangentspace{u} \cap \tangentspace{\ci u}$. We conclude that $\lambda_1=\lambda$ must be simple and 
\begin{align}
\label{positive}
\langle E''(u) \, v , v \rangle 
\, \geq \, \lm_2 \, \| v \|^2_{\0} \, \qquad \mbox{for all } v \in \tangentspace{u} \cap \tangentspace{\ci u}.
\end{align}
As \eqref{energy_function-clone} and \eqref{sd} imply $\langle E''(u) \, v , v \rangle \geq \|v\|^2_{\sR}$, 
we can show the coercivity by a case distinction for arbitrary $v \in \tangentspace{u} \cap \tangentspace{\ci u}$:
\begin{itemize}
\item If $ \|v\|^2_{\sR} >  2 \, \lm \|v \|^2_{\0} $, then $ - \lm \|v \|^2_{\0}  > - \tfrac{1}{2} \|v\|^2_{\sR} $ and therefore
$$ \langle E''(u) \, v , v \rangle - \lambda \, (v ,v)_{L^2(\D)} \, \geq \,  \|v\|^2_{\sR} - \tfrac{1}{2} \| v \|^2_{\sR} =  \tfrac{1}{2} \|v\|^2_{\sR}.
$$
\item If $  \|v\|^2_{\sR} \leq 2 \, \lm \|v \|^2_{\0} $, then $\|v\|^2_{\0} \geq \tfrac{1 }{2 \lm} \|v\|^2_{\sR} \, .$ This together with \eqref{positive} yields
\begin{align*}
 \langle E''(u) \, v , v \rangle - \lambda \, (v ,v)_{L^2(\D)} \,  \geq \, (\lm_2 - \lm) \, \| v \|^2_{\0} \, \geq \, \tfrac{1}{2} \left( \tfrac{ \lm_2 -\lm}{  \lm}\right) \|v \|^2_{\sR} \, . 
\end{align*}
\end{itemize}
As $\lm = \lm_1$, the equivalence of the $\|\cdot \|_{\1}$- and the $\| \cdot \|_{\sR}$-norm yields the result.
\end{proof}
\section{Energy-adaptive Riemannian gradient method} 
\label{section-DNSGF}
In the following we will introduce the energy-adaptive Riemannian gradient method for the Gross--Pitaevskii energy in a rotating frame and present our main results regarding its convergence.
\subsection{Method formulation}
To derive the scheme, we will use the concept of Sobolev gradients \cite{Neu97}. In general, a Riemannian gradient method for an energy functional $E$ on the $L^2$-sphere $\mathbb{S}$ is based on the negative Riemannian gradient $- P_{v,X} (\,\nabla_X E(\,v\,) \,)$, which shows in the direction of the {\it Riemannian steepest descent} in a point $v \in \mathbb{S}$. The Riemannian gradient is determined in two steps as follows. First, we select a Hilbert space $X$ with inner product $(\cdot,\cdot)_X$ and represent the regular gradient $E^{\prime}(v)$ with respect to the $X$-metric, i.e., the Riesz-representation $\nabla_X E(v) \in X$ is defined as the solution to
\begin{align*}
(\nabla_X E(v) , w )_X = \langle E^{\prime}(v) , w \rangle \qquad
\mbox{for all } w \in X.
\end{align*}
The above representation of $E^{\prime}(v)$ is called a Sobolev gradient. We refer to \cite{DaK10,DaP17,HeP20,KaE10} for common choices for $X$ in the context of the GPEVP. It is also possible to equip the space $X$ adaptively with weighted $H^1$-inner products where the optimal weight functions are computed from nonlinear optimization problems (involving $E^{\prime\prime}$) as proposed in \cite{NoPr23}.

In general, the Riemannian gradient is obtained from the Sobolev-gradient $\nabla_X E(v)$ by projecting it into the tangent space $\tangentspace{v}$. This yields the relevant component of the direction $\nabla_X E(\,v\,)$ that is tangential to $\mathbb{S}$ (in $v$). To be precise, $P_{v,X} (\,\nabla_X E(v) \,) \in \tangentspace{v}$ fulfills
\begin{align*}
( P_{v,X} (\,\nabla_X E(v) \,) , w )_X = ( \nabla_X E(v) , w )_X 
\qquad \mbox{for all } w \in \tangentspace{v}.
\end{align*}
The negative Riemannian gradient $- P_{v,X} (\,\nabla_X E(v) \,)$ defines the (Riemannian) steepest descent direction in $v \in \mathbb{S}$ with respect to the $X$-metric. With this, the corresponding Riemannian gradient method is of the form
\begin{align}
\label{riemannian-gradient-method}
u^{n+1} = \frac{ \hspace{-24pt}u^n - \tau_n \ P_{u^n,X} (\,\nabla_X E(\,u^n\,) \,) }{\| u^n - \tau_n \ P_{u^n,X} (\,\nabla_X E(\,u^n\,) \,) \|_{L^2(\D)} },
\end{align}
i.e., from $u^n \in \mathbb{S}$ we move by the distance $\tau_n>0$ into the steepest descent direction. Typically, $\tau_n$ is computed such that the energy dissipation per iteration becomes as large as possible. After each step, the result has to be normalized in $L^2(\D)$ (which is the canonical retraction to $\mathbb{S}$). The abstract scheme \eqref{riemannian-gradient-method} can be also seen as a discrete gradient flow on $\mathbb{S}$, where the choice $X=L^2(\D)$ yields the projected $L^2$-Sobolev gradient flow, cf. \cite{ALT17,BaC13b}, the choice $X=H^1_0(\D)$ with $H^1$-inner product yields the projected $H^1$-Sobolev gradient flow, cf. \cite{CLLZ24-discrete,CLLZ24,KaE10}, and the choice $X=H^1_0(\D)$ with the $(\cdot,\cdot)_{\sR}$-inner product yields the projected $a_0$-Sobolev gradient flow, cf. \cite{CLLZ24,DaK10}.

In \cite{HeP20} it was suggested to select the Hilbert space $X$ adaptively in each iteration such that the Sobolev gradient becomes the identity (in the spirit of optimal preconditioning). To achieve this, we use the inner product that is induced by the linearized Gross--Pitaevskii operator defined in \eqref{defAcal-u-square}. Hence, for any $u\in \mathbb{S}$ we define
\begin{align}
\label{a_z-norm}
\langle \Acal_{|u|} v , w \rangle \,\,=\,\, \re \int_{\D} \nabla_{\sR}v \cdot \overline{\nabla_{\sR}w} + \VR v \, \overline{w} + \beta |u|^2 v \, \overline{w} \dx ,
\end{align}
for $v,w \in H^1_0(\D)$. Equipping $H^1_0(\D)$ with the $\langle \Acal_{|u|} \cdot,\cdot\rangle$-inner product, a simple calculation shows that
\begin{align*}
\langle \Acal_{|u|}  \nabla_X E(u) , v \rangle 
= \langle E^{\prime}(u) , v \rangle = 
\langle \Acal_{|u|} u , v \rangle
\qquad \mbox{for all } v\in H^1_0(\D)
\end{align*}
and hence $ \nabla_X E(u)=u$ as desired. Furthermore, the $\langle \Acal_{|u|} \, \cdot , \cdot \rangle$-orthogonal projection $P_{u,X} : H^1_0(\D) \rightarrow \tangentspace{u}$ is given by
\begin{align}
\label{def-ortho-projection}
P_{u,X}(v) = v - \frac{(u,v)_{\0}}{(u, \Acalinv{u} u)_{\0} } \Acalinv{u} u,
\end{align}
where $\mathcal{I} : \3 \rightarrow H^{-1}(\D)$ is the canonical identification $\mathcal{I}v :=(v, \cdot )_{\0}$. Note that this means that $\Acalinv{u}u =: z\in \3$ is well-defined as the unique solution to
\begin{align*}
\langle \Acal_{|u|} z , v \rangle
= (u ,v)_{\0} 
\qquad\mbox{for all } v \in \3.
\end{align*}
Altogether, using $\nabla_{X} E(u)=u$ and \eqref{def-ortho-projection} in the general definition of the Riemannian gradient method \eqref{riemannian-gradient-method}, we obtain the following scheme.
\begin{definition}[Energy-adaptive Riemannian gradient method]
Let $u^0 \in \mathbb{S}$ denote a starting value and $\{ \tau_n \}_{n\in \mathbb{N}}$ a sequence of positive step lengths. Then for $n \geq 0$, the iterate $u^{n+1} \in \mathbb{S}$ is defined as
\begin{align}
\label{method}
u^{n+1} = \frac{ \hspace{-29pt} (1-\tau_n) u^n + \tau_n \gamma^n \Acalinv{u^n}u^n }{ \| (1-\tau_n) u^n + \tau_n \gamma^n \Acalinv{u^n}u^n \|_{\0}} \qquad \mbox{with} \qquad \gamma^n := ( u^n , \Acalinv{u^n}u^n )_{\0}^{-1}.
\end{align}
\end{definition}
In \cite{HeP20} for $\Omega=0$, the above method was called projected $\boldsymbol{a_{z}}$-Sobolev gradient flow. The iterations \eqref{method} can be also interpreted as a generalized inverse iteration with damping, where $\tau_n$ represents the damping parameter. Every iteration requires the solution of a linear elliptic problem with the linearized GP-operator $\Acal_{|u^n|}$.
\begin{remark}[Optimal step length]
In practice, the step lengths $\tau_n>0$ in the Riemannian gradient method are selected adaptively by line search such that
\begin{align}
\label{tau-n-optimal}
\tau_n = \underset{\tau >0}{\mbox{\normalfont arg\,min}} \, E (\, u^{n+1}(\tau)\,),
\end{align}
where $u^{n+1}(\tau) \in \mathbb{S}$ denotes the approximation in \eqref{method} for $\tau_n=\tau$. Since $E (\, u^{n+1}(\tau)\,)$ is a rational function in $\tau$ that can be cheaply evaluated, the cost for a line search are negligibly small. 
 The efficient realization is explained in detail in Appendix \ref{appendix-B}.
\end{remark}

\begin{remark}[Continuous gradient flow]
As in the case without rotation \cite{HeP20}, the gradient method \eqref{method} can be also seen as a semi-explicit backward Euler discretization of a continuous gradient flow. The corresponding gradient flow in our setting seeks $u \in C^1((0,\infty);H^1(\D))$ with
\begin{align}
\label{gradient-flow}
u^{\prime}(t) = - u(t) + \frac{ \| u(t) \|_{L^2(\D)}^2 }{ ( u(t) , \Acalinv{u(t)} u(t) )_{\0} ) } \Acalinv{u(t)} u(t),
\end{align}
where the contribution $ \| u(t) \|_{L^2(\D)}^2$ does not appear in \eqref{method}, because the term is treated explicitly and hence becomes $\| u^n \|_{L^2(\D)}^2=1$. To convince ourselves that the (unusual) gradient flow \eqref{gradient-flow} remains naturally on $\mathbb{S}$ and diminishes the energy over time, we start with an initial value $u(0)=u_0 \in \mathbb{S}$. In this case, we obtain
\begin{eqnarray*}
\tfrac{1}{2} \frac{\mbox{\normalfont d}}{\mbox{\normalfont d}t} \| u(t) \|_{L^2(\D)}^2 
&=& ( u^{\prime}(t) , u(t) )_{\0} \\
&\overset{\eqref{gradient-flow}}{=}& - \| u(t) \|_{L^2(\D)}^2  +  \frac{ \| u(t) \|_{L^2(\D)}^2 }{ ( u(t) , \Acalinv{u(t)} u(t) )_{\0} ) } ( \Acalinv{u(t)} u(t) , u(t) )_{\0} \,\,=\,\, 0.
\end{eqnarray*}
Hence, $\| u(t) \|_{L^2(\D)}^2$ is constant in time and therefore $\| u(t) \|_{L^2(\D)}^2=\| u(0) \|_{L^2(\D)}^2=1$. The gradient flow remains on $\mathbb{S}$ for all times. Since we just verified $( u^{\prime}(t) , u(t) )_{\0}=0$ and $ u(t) \in \mathbb{S}$, we can write $\gamma_{u(t)} := \frac{ \| u(t) \|_{L^2(\D)}^2 }{ ( u(t) , \Acalinv{u(t)} u(t) )_{\0} ) } = ( u(t) , \Acalinv{u(t)} u(t) )_{\0} )^{-1}$ and obtain the energy dissipation as
\begin{eqnarray*}
0 &\le& \langle \Acal_{|u(t)|} u^{\prime}(t) , u^{\prime}(t) \rangle
) \,\,=\,\, -  \langle \Acal_{|u(t)|} u(t) , u^{\prime}(t) \rangle + \gamma_{u(t)} \,  \langle \Acal_{|u(t)|} \Acalinv{u(t)} u(t) , u^{\prime}(t) \rangle\\
&=& -   \langle \Acal_{|u(t)|} u(t) , u^{\prime}(t) \rangle +\gamma_{u(t)}  ( u(t) , u^{\prime}(t))_{L^2(\D)} 
\,\,=\,\, -   \langle \Acal_{|u(t)|} u(t) , u^{\prime}(t) \rangle \,\,=\,\, - \frac{\mbox{\normalfont d}}{\mbox{\normalfont d}t} E(u(t)),
\end{eqnarray*}
i.e., $E(u(t))$ is monotonically decreasing in time. In conclusion, the energy-adaptive gradient method \eqref{gradient-flow} corresponds to a time discretization of an energy-dissipative gradient flow on $\mathbb{S}$. This gives an alternative motivation for the numerical method.
\end{remark}

\subsection{Main convergence results}
In the following we present two types of main results. The first one guarantees that the scheme is globally energy-diminishing, which in turns implies global convergence of the densities $|u^n|$ to, at least, a critical point of the Gross--Pitaevskii functional. The second results quantifies the local convergence behavior in neighborhoods of ground states.
\begin{theorem}[Global energy-dissipation]
\label{global_convergence_theorem}
Assume \ref{A1}-\ref{A4} and let $\tau_{\mbox{\rm\tiny min}} >0$ be a sufficiently small lower bound for the step sizes (to exclude $\tau_n$ degenerating to zero). Then, there exists an upper bound $ \tau_{\mbox{\rm\tiny max}}>\tau_{\mbox{\rm\tiny min}}$, such that for any $\tau_n \in [\tau_{\mbox{\rm\tiny min}} ,\tau_{\mbox{\rm\tiny max}}]$, the iterations $\{u^n \}_{n \in \mathbb{N}}$ generated by Riemannian gradient method \eqref{method} have the following properties:
\begin{itemize}
\item[(i)] It holds $E(u^n)-E(u^{n+1}) \geq C_{\tau} \, \|u^{n+1} - u^n \|^2_{\1}$ for all $n\ge 0$ and a constant $C_{\tau}>0$ (that depends on $\tau_{\mbox{\rm\tiny min}}$, $ \tau_{\mbox{\rm\tiny max}}$ and which fulfills $C_{\tau}\ge \tfrac{1}{2}$ for $\tau_n$ small enough).
\item[(ii)] There exists a limit energy $\boldsymbol{E}:=\lim\limits_{n \rightarrow \infty} E(u^n)$.
\item[(iii)] There is a subsequence $\{ u^{n_j} \}_{j \in \mathbb{N}}$ of $\{ u^n \}_{n \in \mathbb{N}}$ and $u \in \mathbb{S}$ such that 
$\lim\limits_{j\rightarrow \infty} \| u^{n_j} - u\|_{H^1(\D)}=0$.
Furthermore, $u$ is an eigenfunction to the GPEVP \eqref{eigenvalue-E-derivative-form}, i.e., 
it holds 
\begin{align*}
 E^{\prime}(u)= \lambda \, \mathcal{I} u
\qquad \mbox{for} \qquad \lm := (u, \Acalinv{u} u )^{-1}_{\0} = \lim_{j \rightarrow \infty} \gamma^{n_j}.\\[-2.3em]
\end{align*}
\item[(iv)] If any of the limits $u$ in (iii) is a locally quasi-unique ground state to \eqref{minimization-problem}, then the whole sequence of densities $|u^n|^2$ converges to the corresponding ground state density $|u|^2$ with
\begin{align*}
\lim_{n\rightarrow \infty} \| \, |u^{n}|^2 - |u|^2 \,\|_{L^2(\D)} = 0 
\qquad
\mbox{and}
\qquad
\lim_{n\rightarrow \infty} \| \, |u^{n}| - |u| \,\|_{H^1(\D)} = 0.
\end{align*}
\end{itemize}
\end{theorem}
The proof of the theorem is postponed to Section \ref{section-global-convergence}.

When reflecting on the results of Theorem \ref{global_convergence_theorem}, recall that locally quasi-unique ground states are only (locally) unique up to multiplicative phase shifts $\exp(\ci \phase)$. Since the iterations \eqref{method} do not lock the phase (with respect to some reference ground state), it is open if the phase of the iterates $u^n$ can potentially \quotes{circulate forever}, which would imply that the full sequence $u^n$ is not converging (but only up to subsequences). However, since the phase is eliminated in the density $|u^n|$ (which we recall as the physical quantity of interest), convergence of the full sequence can be guaranteed now, cf. Theorem \ref{global_convergence_theorem} (iv). This aspect is again resembled in the local convergence result in Theorem \ref{linear_rate_theorem} where convergence is established up to phase shifts (or for the density itself). Before we present the result, we need to introduce a weighted eigenvalue problem which will allow us to characterize the local convergence rates. As we will see in Section \ref{section-local-convergence-proof}, the weighted eigenvalue problem is obtained after interpreting \eqref{method} as fixed-point iteration $\phi_{\tau}(u)=u$ and by deriving a characterization of the spectrum of $\phi'_{\tau}(u)$.
\begin{lemma}\label{lemma-weighted-eigenvalueproblem}
Assume \ref{A1}-\ref{A4} and let $u \in \mathbb{S}$ be a locally quasi-unique ground state with corresponding ground state eigenvalue $\lambda$. For that, consider the weighted linear eigenvalue problem seeking $v_i \in \5$ with $\| v_i \|_{L^2(\D)}=1$ and $\mu_i \in \mathbb{R}$ (for $i \in \mathbb{N}$) such that
\begin{align}
\label{weighted-evp1}
( \lm v_i - 2\beta \re(u \overline{v_i} ) u ,w )_{L^2(\D)} = \mu_i \, \langle \Acal_{|u|} v_i , w \rangle \qquad \mbox{for all } w\in  \5.
\end{align}
Then, for all eigenvalues $\mu_i$, it holds
\begin{align}
\label{boundes_mu_i}
- 1 < \frac{-\lambda_1}{\lambda_1+\delta_1} \le \mu_i \le \frac{\lambda_1}{\lambda_2} < 1,
\end{align}
 where $\lambda_1 = \lambda$ is the smallest and $\lambda_2$ the second smallest eigenvalue of $E^{\prime\prime}(u)\vert_{\tangentspace{u}}$ (cf. \eqref{spectrum-E''u}) and $\delta_1>0$ is the smallest eigenvalue of $\Acal_{0}$ (i.e. the operator $\Acal_{|v|}$ for $|v|=0$).
\end{lemma}
\begin{proof}
Noting that 
$$
\langle \Acal_{|u|} v_i , v_i  \rangle - \langle (E''(u)-\lm \mathcal{I}) v_i , v_i \rangle = \lm - 2\beta ( \re(u \overline{v_i} ) u ,v_i )_{L^2(\D)}
$$ 
and testing in \eqref{weighted-evp1} with $w=v_i$ yields
\begin{align*}
1-\mu_i = \tfrac{\langle (E''(u)-\lm \mathcal{I}) v_i , v_i \rangle }{ \langle \Acal_{|u|} v_i , v_i \rangle} \overset{\eqref{coercivity-secE-lambda}}{>}0.
\end{align*}
This directly implies $\mu_i < 1$ for all $i \in \mathbb{N}$. This bound can now be sharpened as
\begin{align*}
\inf_{i \in \mathbb{N}} \,(1 - \mu_i ) 
&\,=\, \inf_{i \in \mathbb{N}}  \frac{\langle (E''(u)-\lm \mathcal{I}) v_i , v_i \rangle }{ \langle \Acal_{|u|} v_i , v_i \rangle} 
\, \ge \, \inf_{i \in \mathbb{N}}  \frac{\langle (E''(u)-\lm \mathcal{I}) v_i , v_i \rangle }{\langle E''(u) v_i , v_i \rangle }
\,=\, \inf_{i \in \mathbb{N}} \left( 1 - \lm \frac{ ( v_i , v_i )_{L^2(\D)} }{\langle E''(u) v_i , v_i \rangle }  \right) \\
&\,=\, 1 -  \lm  \sup_{i \in \mathbb{N}} \frac{ ( v_i , v_i )_{L^2(\D)} }{\langle E''(u) v_i , v_i \rangle } 
\,=\, 1 -  \lm  \sup_{w \in \5} \frac{ ( w , w )_{L^2(\D)} }{\langle E''(u) w , w \rangle } 
\,=\, 1 - \frac{  \lm }{ \lambda_2 } .
\end{align*}
In the last step we used the min-max principle for linear self-adjoint operators, recalling that $\ci u \in \tangentspace{u}$ is an eigenfunction of $E^{\prime\prime}(u)\vert_{\tangentspace{u}}$ to the smallest eigenvalue $\lm_1=\lm$, and therefore $\lambda_2^{-1}=\sup\limits_{w \in \5} \frac{ ( w , w )_{L^2(\D)} }{\langle E''(u) w , w \rangle }$, with $\lambda_2$ being the second smallest eigenvalue of $E^{\prime\prime}(u)\vert_{\tangentspace{u}}$.

It remains to check that $-1$ is a lower bound for $\mu_i$. For this, we first need to establish the pointwise inequality $\lambda \ge \beta |u|^2$ as done in \cite[Lemma 3.6]{PH24} for $\Omega=0$. Here, $\lambda$ denotes the ground state eigenvalue to the ground state $u$. Without loss of generality we assume that $V \in C^{\infty}(\overline{\D})$ such that $u \in C^2(\D)$ (the general result for $V\in L^{\infty}(\D)$ follows by density arguments, cf. \cite[Lemma 3.6]{PH24}). Since $|u|^2\vert_{\partial \D} =0$ and $\| u\|_{L^2(\D)}=1$, we know that $|u|^2$ has an interior maximum $|u^{\ast}|^2=|u(x^{\ast})|^2>0$ for some point $x^{\ast} \in \D$. Due to the properties of an interior maximum, it holds $-\Delta |u^{\ast}|^2 := -\Delta |u(x)|^2\vert_{x=x^{\ast}} \ge 0$. From the GPEVP \eqref{eigen_value_problem_1}, we have 
\begin{align*}
\lambda \, u^{\ast} = -\Delta u^{\ast} + V(x^{\ast})\, u^{\ast} - \Omega \, \mathcal{L}_{3} u^{\ast}\vert_{x^{\ast}} + \beta \, |u^{\ast}|^2 u^{\ast}.
\end{align*}
Multiplying equality with $\overline{u^{\ast}}$ and afterwards the complex conjugate of the equality with $u^{\ast}$ yields, respectively
\begin{align*}
 \lambda \, |u^{\ast}|^2 &= - \overline{u^{\ast}} \Delta u^{\ast} + V(x^{\ast})\, |u^{\ast}|^2 - \Omega \, \overline{u^{\ast}}\, \mathcal{L}_{3} u^{\ast}\vert_{x^{\ast}} + \beta \, |u^{\ast}|^4,\\
 \lambda \, |u^{\ast}|^2 &= - u^{\ast} \overline{\Delta u^{\ast}} + V(x^{\ast})\, |u^{\ast}|^2 - \Omega \, u^{\ast}\, \overline{\mathcal{L}_{3} u^{\ast}\vert_{x^{\ast}}} + \beta \, |u^{\ast}|^4.
\end{align*}
Summing up the two equations and using $2 \re(a \overline{b}) = a \overline{b} + \overline{a} b$ yields
\begin{align}
\label{eqn-uast}
 \lambda \, |u^{\ast}|^2 &= - \re( \overline{u^{\ast}} \Delta u^{\ast}) + V(x^{\ast})\, |u^{\ast}|^2 - \Omega \re( \overline{u^{\ast}}\, \mathcal{L}_{3} u^{\ast}\vert_{x^{\ast}}) + \beta \, |u^{\ast}|^4.
\end{align}
Next, note that
\begin{align*}
-\Delta |u^{\ast}|^2 = - 2 \re( \overline{u}\, \Delta u )  - 2 |\nabla u|^2
\quad
\mbox{and}
\quad |\Omega \, \re( \overline{u^{\ast}}\, \mathcal{L}_{3} u^{\ast}\vert_{x^{\ast}})| \le \tfrac{\Omega^2}{4}  (|x_1^{\ast}|^2+|x_2^{\ast}|^2) |u^{\ast}|^2 + |\nabla u^{\ast}|^2.
\end{align*}
Hence, together with $-\Delta |u^{\ast}|^2\ge 0$, we obtain from \eqref{eqn-uast}
\begin{align*}
 \lambda \, |u^{\ast}|^2 &= - \tfrac{1}{2} \Delta |u^{\ast}|^2 + |\nabla u^{\ast}|^2 + V(x^{\ast})\, |u^{\ast}|^2 - \Omega \re( \overline{u^{\ast}}\, \mathcal{L}_{3} u^{\ast}\vert_{x^{\ast}}) + \beta \, |u^{\ast}|^4,\\
 &\ge \left(  V(x^{\ast})  - \tfrac{\Omega^2}{4}  (|x_1^{\ast}|^2+|x_2^{\ast}|^2) \right) |u^{\ast}|^2 \,+\, \beta \, |u^{\ast}|^4 \\
 &\overset{\ref{A4}}{\ge} K \tfrac{\Omega^2}{4}  (|x_1^{\ast}|^2+|x_2^{\ast}|^2)  |u^{\ast}|^2  \,+\, \beta \, |u^{\ast}|^4 \,\, \ge \,\,  \beta \, |u^{\ast}|^4.
\end{align*}
Since $|u^{\ast}|^2 = \| u \|_{L^{\infty}(\D)}^2>0$, we obtain $\lambda \ge   \beta \| u \|_{L^{\infty}(\D)}^2$ as desired. We are now ready to prove the lower bound, which we obtain as
\begin{eqnarray*}
\lefteqn{1-\mu_i = \frac{\langle (E''(u)-\lm \mathcal{I}) v_i , v_i \rangle }{ \langle \Acal_{|u|} v_i , v_i \rangle} 
\le \frac{\langle (E''(u) v_i , v_i \rangle - \beta   \| u \|_{L^{\infty}(\D)}^2}{ \langle \Acal_{|u|} v_i , v_i \rangle} }\\
&\le& 1 + \frac{ 2\, \beta \| \re(u \overline{v_i})\|^2_{\0} - \beta   \| u \|_{L^{\infty}(\D)}^2 }{ \|v_i\|_{\sR}^2 + \beta \|u \overline{v_i}\|_{\0}^2}
\,\,\le\,\, 1 + \frac{ \beta \| u \overline{v_i} \|^2_{\0} }{ \|v_i\|_{\sR}^2 + \beta \|u \overline{v_i}\|_{\0}^2}
\\
& \le& 1 + \frac{ \beta \| u \overline{v_i} \|^2_{\0} }{ \delta_1 \|v_i\|_{L^2(\D)}^2 + \beta \|u \overline{v_i}\|_{\0}^2}
\,\, \le \,\, 1 + \frac{ \beta \| u \overline{v_i} \|^2_{\0} }{ \delta_1 \tfrac{\beta}{\lambda} \| u \overline{v_i}\|^2_{\0} + \beta \|u \overline{v_i}\|_{\0}^2}
\,\, = \,\,  1 + \frac{\lambda}{\delta_1 + \lambda} < 2,
\end{eqnarray*}
where we used in the penultimate step that
\begin{align*}
\tfrac{\beta}{\lambda} \| u \overline{v_i}\|^2_{\0} \le \tfrac{\beta}{\lambda} \| u\|^2_{L^{\infty}(\D)} \le 1 =  \|v_i\|_{L^2(\D)}^2.
\end{align*}
\end{proof}

With the above result, we can now quantify the local convergence, where we find that the density $|u^n|^2$ converges locally with linear rate to $|u|^2$ in $L^2(\D)$ and $u^n$ converges locally linear in $H^1(\D)$ to $u$, up to a constant phase factor.
\begin{theorem}[Local convergence]
\label{linear_rate_theorem}
Assume \ref{A1}-\ref{A4} and let $u \in \mathbb{S}$ denote a locally quasi-unique ground state to \eqref{minimization-problem}. We consider the iterations $\{u^n \}_{n \in \mathbb{N}}$ of the Riemannian gradient method \eqref{method} for a uniform step size $\tau_n = \tau >0$. For $\mu_i$ denoting the eigenvalues to \eqref{weighted-evp1}, the contraction constant fulfills
\begin{align*}
\rho^{\ast} := \max_{i \in \mathbb{N} }\, | (1-\tau) + \tau \mu_i | \,<\, 1 \qquad \mbox{for all } 0 < \tau \le 1. 
\end{align*}
More precisely, if $\mu_1$ denotes the maximum eigenvalue in magnitude to \eqref{weighted-evp1}, then the following is true:
\begin{itemize}
\item If $\mu_1>0$, it holds 
$\rho^{\ast} < 1$ for all $0<\tau < 1 + \frac{\lambda_2-\lambda_1}{\lambda_2+\lambda_1}$ and we have $|\mu_1| \le \frac{\lambda_1}{\lambda_2}<1$. 
\item If $\mu_1<0$, it holds $\rho^{\ast} < 1$ for all 
$0<\tau <  1 + \frac{\delta_1}{2\lambda_1+\delta_1}$ and we have $|\mu_1| \le \tfrac{\lambda_1}{\lambda_1+\delta_1}<1$
\end{itemize}
where $\delta_1>0$ is as in Lemma \ref{lemma-weighted-eigenvalueproblem}.
Furthermore, for each $\eps>0$, there is an $H^1$-neighborhood $S_{\eps}$ of $u$ on $\mathbb{S}$ and a constant $C_{\eps}>0$ such that for all starting values $u^0 \in S_{\eps}$ we have locally linear convergence for the density
\begin{align*}
\|  |u^{n}|^2 - |u|^2 \|_{L^2(\D)} &\le C_{\eps} \hspace{2pt}\, |\rho^* + \eps|^{n} \,\hspace{2pt} \| u^0 - u \|_{\1} \, 
\end{align*}
and locally linear convergence up to phase shifts for $u^n$
\begin{align}
\label{rate-linear-standard}
\inf_{\phase_n \in [-\pi,\pi)} 
\|  u^{n}- \exp(\ci \phase_n) \, u \|_{\1} &\le C_{\eps} \hspace{2pt}\, |\rho^* + \eps|^{n} \,\hspace{2pt} \| u^0 - u \|_{\1}.
\end{align}
Recall here that $ \exp(\ci \phase_n) \, u$ is a ground state that is equivalent to $u$ in the sense that they have identical ground state densities $|u|^2$ and the only difference is a constant phase shift.
 \end{theorem}
Theorem \ref{linear_rate_theorem} implies guaranteed local convergence for any $\tau \le 1$. Note that in particular in combination with \eqref{boundes_mu_i}, the choice $\tau=1$ yields convergence of the inverse iteration method $u^{n+1} = \tfrac{ \Acalinv{u^n}u^n }{\| \Acalinv{u^n}u^n \|_{\0}}$, which converges, if $\mu_1>0$, with the rate
\begin{align}
\label{convergence-inverse-iteration}
\inf_{\phase_n \in [-\pi,\pi)} \|  u^{n}- \exp(\ci \phase_n) \, u \|_{\1} &\le C \hspace{2pt}\, |\tfrac{\lambda_1}{\lambda_2} + \eps  |^{n} \,\hspace{2pt} \| u^0 - u \|_{\1}.
\end{align}
For linear problems, this estimate coincides with the well-known result that the convergence of the inverse iteration depends on the size of the first spectral gap.

\begin{remark}[Generalization to local minimizers]
Unlike the non-rotating case (cf. \cite{PH24}), where the analysis depends on the fact that the ground state eigenvalue is the smallest eigenvalue of the linearized Gross--Pitaevskii operator $\mathcal{A}_{|u|}$ (which is no longer true in the rotating setting), our analysis only exploits the properties of the spectrum of $E^{\prime\prime}(u)\vert_{\tangentspace{u}}$. Since any local minimizer $\hat{u}$ (that is locally quasi-unique) also satisfies the sufficient second-order optimality conditions, all results of Theorem \ref{linear_rate_theorem} remain valid for quasi-unique local minimizers if we replace $E^{\prime\prime}(u)\vert_{\tangentspace{u}}$ by $E^{\prime\prime}(\hat{u})\vert_{\tangentspace{\hat{u}}}$.
\end{remark}
\section{Proof of Theorem \ref{global_convergence_theorem}: Global convergence results}
\label{section-global-convergence}
In this section we prove Theorem \ref{global_convergence_theorem}, where we exploit similar arguments as in \cite{HeP20} for $\Omega=0$. For brevity, we write the approximation $u^{n+1} \in \mathbb{S}$ in \eqref{method} as $u^{n+1} = \tfrac{\hat{u}^{n+1}}{\|\hat{u}^{n+1}\|_{\0}}$, where $\hat{u}^{n+1} := (1-\tau_n) u^n + \tau_n \gamma^n \Acalinv{u^n}u^n$.
\subsection{Global energy dissipation}
We start with proving the global energy dissipation. The first step is to see that the preliminary iterates $\hat{u}^{n+1}$ increase the $L^2$-norm and hence $E(u^{n+1}) \le E(\hat{u}^{n+1})$ after normalization.
\begin{lemma}
\label{lemma-inter-mass-growth}
Assume \ref{A1}-\ref{A4}. For all $n\ge 0$ and $u^n\in \mathbb{S}$ the $L^2$-norm grows with
\begin{align}
\label{normalization-error}
 \| \hat{u}^{n+1} \|_{\0} - \| u^n \|_{\0} 
= \frac{( u^n - u^{n+1} , u^n)_{\0}}{ ( u^{n+1} , u^n)_{\0}}  > 0.
\end{align}
Furthermore, we have $\| \hat{u}^{n+1} \|_{\0}= \| u^{n} \|_{\0}$ if and only if $u^n = \hat{u}^{n+1}$.
\end{lemma}
\begin{proof}
From $\tfrac{1}{\tau_n}(\hat{u}^{n+1} - u^n)  = - u^n + \gamma^n \Acalinv{u^n}u^n$ we obtain for $u^{n}\in \mathbb{S}$
\begin{align}
\label{mass-equality-1}
\frac{1}{\tau_n} ( \hat{u}^{n+1} - u^n, u^n)_{\0} = - \| u^n \|_{\0}^2 + (u^n,\Acalinv{u^n} u^n)_{\0}^{-1} (\Acalinv{u^n}  u^n, u^n )_{\0} = 0.
\end{align}
Hence, $1=\| u^{n} \|_{\0}^2 \le  \| u^{n} \|_{\0}^2 + \| \hat{u}^{n+1} \hspace{-1pt}-\hspace{-1pt} u^n\|_{\0}^2  = \| \hat{u}^{n+1} \|_{\0}^2$.
Identity \eqref{normalization-error} follows as
\begin{align*}
\frac{( u^n , u^n)_{\0}}{ ( u^{n+1} , u^n)_{\0}}
= \| \hat{u}^{n+1} \|_{\0}  \frac{( u^n , u^n)_{\0}}{ ( \hat{u}^{n+1} , u^n)_{\0}} 
\overset{\eqref{mass-equality-1}}{=} \| \hat{u}^{n+1} \|_{\0}.
\end{align*}
\end{proof}
Next, we prove a preliminary lower bound for the energies.
\begin{lemma}
\label{energy_error}
Assume \ref{A1}-\ref{A4} and let $\tau_n \leq \tfrac{1}{2}$, then
\begin{align}
\label{energy_estimate_0}
E(u^n) - E(\hat{u}^{n+1}) \geq -  \int_{\D} \tfrac{3 \beta}{4}|\hat{u}^{n+1}- u^{n}|^4 + \bigl( \tfrac{1}{\tau_n} - \tfrac{1}{2} \bigr) \, \|\hat{u}^{n+1} - u^n\|_{\sR}^2  .
\end{align}
If $ \tau_n \geq 2$, the iterates diverge, that is, either $E(\hat{u}^{n+1}) > E(u^n)$ or $u^n=\hat{u}^{n+1}$.
\end{lemma}
\begin{proof}
With $\hat{u}^{n+1} = (1-\tau_n) u^n + \tau_n \gamma^n \Acalinv{u^n}u^n$ and \eqref{mass-equality-1} we have
\begin{eqnarray}
\label{pseudo-energy-identity}
\frac{1}{\tau_n} \langle \Acal_{|u^n|} (\hat{u}^{n+1} - u^n) , \hat{u}^{n+1} - u^n \rangle 
&=&
 - \langle \Acal_{|u^n|} u^n , \hat{u}^{n+1} - u^n \rangle
\end{eqnarray}
and hence
\begin{eqnarray*}
\lefteqn{ \langle \Acal_{|u^n|} (\hat{u}^{n+1} - u^n) , \hat{u}^{n+1} - u^n \rangle
\,\,\, = \,\,\,
\langle \Acal_{|u^n|} \hat{u}^{n+1}, \hat{u}^{n+1} \rangle
- 2 \langle \Acal_{|u^n|} u^{n} , \hat{u}^{n+1} - u^n \rangle
- \langle \Acal_{|u^n|} u^{n} , u^{n} \rangle } \\
&\overset{\eqref{pseudo-energy-identity}}{=}&
\langle \Acal_{|u^n|} \hat{u}^{n+1}, \hat{u}^{n+1} \rangle
+ \tfrac{2}{\tau_n}\langle \Acal_{|u^n|} (\hat{u}^{n+1} - u^{n})  , \hat{u}^{n+1} - u^n \rangle
- \langle \Acal_{|u^n|} u^{n} , u^{n} \rangle.\hspace{70pt}
\end{eqnarray*}
We conclude
\begin{eqnarray*}
\langle \Acal_{|u^n|} u^{n} , u^{n} \rangle
&=&
\langle \Acal_{|u^n|}\hat{u}^{n+1}, \hat{u}^{n+1} \rangle 
+ \bigl(\tfrac{2}{\tau_n} - 1\bigr) \langle \Acal_{|u^n|}( \hat{u}^{n+1} - u^{n} ) , \hat{u}^{n+1} - u^n \rangle.
\end{eqnarray*}
Together with 
$
\langle \Acal_{|u^n|} u^{n} , u^{n} \rangle = 2 E(u^n) + \tfrac{\beta}{2} \int_{\D} |u^{n}|^4 \dx
$
we obtain
\begin{eqnarray}
 \nonumber \lefteqn{E(u^n)  - E(\hat{u}^{n+1} ) 
\,\,\,=\,\,\,
- \tfrac{\beta}{4} \int_{\D} (|\hat{u}^{n+1}|^2 - |u^{n}|^2)^2 \dx
+ \bigl(\tfrac{1}{\tau_n} - \tfrac{1}{2}\bigr) \langle \Acal_{|u^n|} (\hat{u}^{n+1} - u^{n})  , \hat{u}^{n+1} - u^n \rangle }\\
\nonumber &=&
- \tfrac{\beta}{4} \int_{\D} (|\hat{u}^{n+1}|+ |u^{n}|)^2  (|\hat{u}^{n+1}| -| u^{n}|)^2\dx 
+ \bigl(\tfrac{1}{\tau_n} - \tfrac{1}{2}\bigr) \beta \int_{\D} |u^n|^2 |\hat{u}^{n+1} - u^{n}|^2\dx \\
 \label{growth_energy} & &+ \bigl(\tfrac{1}{\tau_n} - \tfrac{1}{2}\bigr) ( \hat{u}^{n+1} - u^{n}  , \hat{u}^{n+1} - u^n )_{\sR},
\end{eqnarray}
which is negative for $\tau_n \ge 2$ and the scheme diverges in this regime. Moreover, we get from \eqref{growth_energy} with $(|\hat{u}^{n+1}|+ |u^{n}|)^2  (|\hat{u}^{n+1}| -| u^{n}|)^2 \leq 3 |\hat{u}^{n+1}- u^{n}|^4 + 6 |\hat{u}^{n+1} - u^{n}|^2 |u^n|^2$ that
\begin{eqnarray*}
 \lefteqn{E(u^n)  - E(\hat{u}^{n+1} ) 
\,\,\,\geq\,\,\, - \tfrac{3 \beta}{4} \int_{\D} |\hat{u}^{n+1}- u^{n}|^4 \dx + \beta \bigl(\tfrac{1}{\tau_n} - 2  \bigr) \int_{\D} |u^n|^2 |\hat{u}^{n+1} - u^{n}|^2 \dx }\\
 &\enspace& \qquad+ \bigl(\tfrac{1}{\tau_n} - \tfrac{1}{2}\bigr) ( \hat{u}^{n+1} - u^{n}  , \hat{u}^{n+1} - u^n )_{\sR} \, ,\hspace{180pt}
\end{eqnarray*}
which  for $\tau_n \leq \tfrac{1}{2}$ gives  $\eqref{energy_estimate_0}$.
\end{proof}
Next, we prove that the energy is reduced in each iteration.
\begin{lemma}[Energy dissipation]
 \label{energy_reduction_lemma}
Assume \ref{A1}-\ref{A4}, then there exists $\tau_{\mbox{\rm\tiny max}}<2$ (depending on $\beta$, $E(u^0)$ and $\D$) such that for all $\tau_n \in (0, \tau_{\mbox{\rm\tiny max}})$ and some constant $C_{\tau_n} \ge \tfrac{1}{2}$ it holds
\begin{align}
\label{Energy-diff-by-H1-diff}
E(u^n) -  E(\hat{u}^{n+1})  \ge C_{\tau_n} \| u^n - \hat{u}^{n+1}  \|^2_{\sR} 
\qquad
\mbox{and hence}
\quad
E(u^{n+1})  \leq E(u^n) \, .
\end{align}
\end{lemma}
\begin{proof} 
The estimate $E(u^{n+1}) \leq E(\hat{u}^{n+1})$, follows from Lemma \ref{lemma-inter-mass-growth}. Next we observe that
\begin{eqnarray}
\label{induction_1} \lefteqn{ \| \hat{u}^{n+1} - u^n\|^2_{\sR} 
\,\,\,\leq\,\,\,  \langle \Acal_{|u^n|} (\hat{u}^{n+1} - u^n), \hat{u}^{n+1} - u^n \rangle \,\,\,\overset{\eqref{pseudo-energy-identity}}{=}\,\,\, - \tau_n \, \langle \Acal_{|u^n|}u^n, \hat{u}^{n+1} - u^n \rangle }\\
\nonumber &=&\,\,\,  \tau_n^2 \, \big( \langle \Acal_{|u^n|} u^n, u^n \rangle - \gamma^n \langle \Acal_{|u^n|}  u^n, \Acalinv{u^n} u^n \rangle \big)
\,\,\, \leq\,\,\, \tau_n^2 \, \langle \Acal_{|u^n|}  u^n, u^n \rangle 
\,\,\, \overset{\eqref{energy_function-clone}}{\le} 4 \, \tau_n^2 \, E(u^n).
\end{eqnarray}
With this, the remaining proof is done by induction.
For $n=0$, \eqref{induction_1} guarantees $\| \hat{u}^1 - u^0 \|^2_{\sR} \le 1$ for $\tau_0^2 \le (4 E(u^0))^{-1}$. Using the Sobolev embedding $\3 \hookrightarrow L^4(\D)$ and the equivalence of the $\| \cdot \|_{\sR}$-norm with the standard $H^1$-norm we therefore get
$ \| \hat{u}^1 - u^{0}\|_{L^4(\D)}^{4}
\leq C \| \hat{u}^1 - u^0 \|^4_{\sR}
\leq C \| \hat{u}^1 - u^0 \|^2_{\sR}$. 
Using this in \eqref{energy_estimate_0}, we conclude (for a possibly different constant $C>0$)
\begin{eqnarray*}
E(u^0)  - E(\hat{u}^1 ) \ge (\tfrac{1}{\tau_0} - \tfrac{1}{2} - C \beta)  \| \hat{u}^1 - u^{0}  \|^2_{\sR}.
\end{eqnarray*}
Hence, there exists 
$\tau_{\mbox{\rm\tiny max}}>0$ (depending on $\beta$, $E(u^0)$ and $\D$) such that for all $\tau_0 \in (0,\tau_{\mbox{\rm\tiny max}})$ we have $C_{\tau_0}:= \tfrac{1}{\tau_0} - \tfrac{1}{2} - C \beta \ge \tfrac{1}{2}$ and $E(u^0)  - E(\hat{u}^1 )  \ge   C_{\tau_0} \| u^0 - \hat{u}^{1}  \|^2_{\sR}$. In particular, $E(u^1)\le E(u^0)$.

To conclude let inductively hold $E(u^{n})  \leq E(u^{0})$ and let $\tau_{\mbox{\rm\tiny max}}>0$ be as in step 1. Then, for all $\tau_n \le \tau_{\mbox{\rm\tiny max}}$ and with \eqref{induction_1}, it still holds $\|{\hat{u}^{n+1}} - u^n \|^2_{\sR} \le 4 \tau_n^2 E(u^n) \le 4 \tau_n^2 E(u^0) \le 1$. Consequently, with $C_{\tau_n}:= \tfrac{1}{\tau_n} - \tfrac{1}{2} - C \beta \ge \tfrac{1}{2}$ we can argue as before to verify $E(u^n)  - E({\hat{u}^{n+1}} )  \ge C_{\tau_n} \| {\hat{u}^{n+1}} - u^{n}  \|^2_{\sR}$. The proof is finished by $E(u^{n+1})  \leq E(\hat{u}^{n+1})$ thanks to \eqref{normalization-error}. Also note that $C_{\tau_n} \rightarrow \infty$ for $\tau_n \rightarrow 0$.
\end{proof}
With this, we are ready to prove the energy dissipation as stated in Theorem \ref{global_convergence_theorem} (i).
\begin{corollary}[Energy dissipation]
\label{goes_to_zero}
Assume \ref{A1}-\ref{A4}. There is $\tau_{\mbox{\rm\tiny max}}=\tau_{\mbox{\rm\tiny max}}(E,u^0,\D)<2$ such that for all $\tau_n \in (0, \tau_{\mbox{\rm\tiny max}})$ and a constant $C_{\tau_n} >0$ (with $C_{\tau_n} \rightarrow \infty$ for $\tau_n \rightarrow 0$) such that
\begin{align*}
 E(u^n)-E(u^{n+1}) & \geq C_{\tau_n} \, \|u^{n+1} - u^n \|^2_{\1}.
\end{align*}
\end{corollary}
\begin{proof}
The result follows from Lemma \ref{energy_reduction_lemma} if we show $ \|u^{n+1} - u^n \|_{\1} \le C \| \hat{u}^{n+1} - u^n  \|_{\sR}$. This is easily seen by recalling that $ 1 = \|u^n \|_{\0} \leq \|\hat{u}^{n+1}\|_{\0}$ and $\|u^n\|_{\1} \leq c \, \sqrt{ E(u^n) } \le c \sqrt{ E(u^0) } =: C_{u^0}$. We obtain
\begin{eqnarray*}
\lefteqn{\|u^{n+1} \narrowminus u^n \|_{\1} \,\leq\, \|\hat{u}^{n+1} \|_{\0} \|u^{n+1} \narrowminus u^n\|_{\1}
 \,=\, \|\hat{u}^{n+1} \|_{\0} \| \frac{\hat{u}^{n+1}}{\|\hat{u}^{n+1} \|_{\0}} \narrowminus  u^n \|_{\1} } \\
 & = &\hspace{-5pt} \| \hat{u}^{n+1}  \narrowminus u^n +u^n \narrowminus \|\hat{u}^{n+1} \|_{\0} u^n \|_{\1} 
 \,\leq\, \| \hat{u}^{n+1} \narrowminus u^n \|_{\1} + \| (\|\hat{u}^{n+1} \|_{\0} \narrowminus 1) u^n \|_{\1} \\
 &\leq &\hspace{-5pt}
 \| \hat{u}^{n+1} \narrowminus u^n \|_{\1} + C_{u^0} \,  \| \hat{u}^{n+1} \narrowminus  u^n \|_{\0} 
   \,\leq \,   (1+ C_{u^0}  ) \| \hat{u}^{n+1} \narrowminus u^n \|_{\1} . 
\end{eqnarray*}
Hence, we have with  Lemma \ref{energy_reduction_lemma}
\begin{align*}
 E(u^n) \narrowminus E(u^{n+1} )  \ge E(u^n) \narrowminus E({\hat{u}^{n+1}} ) \geq  C_{\tau_n} \| {\hat{u}^{n+1}} \narrowminus u^{n}  \|^2_{\sR}  \geq C_{\tau_n} (1+ C_{u^0})^{-2}  \|u^{n+1} \narrowminus u^n \|^2_{\1}.
\end{align*}
\end{proof}
\subsection{Proof of Global convergence}
It remains to prove (ii)-(iv) from Theorem \ref{global_convergence_theorem} which we will do in the following.
\begin{proof}[Proof of Theorem \ref{global_convergence_theorem}] 
The first part of the proof follows standard compactness results analogously to \cite{HeP20}. Statement (i) is proved by Corollary \ref{goes_to_zero}, which also implies (ii) since $\{ E(u^n) \}_{n\in\mathbb{N}}$ is a monotonically decreasing sequence of real numbers. As $E(u^n)\le E(u^0)$, the sequence $\{u^n\}_{n\in \mathbb{N}}$ is uniformly bounded  in $\3$ and (by Rellich--Kondrachov embedding) we can extract a subsequence $\{ u^{n_j} \}_{j \in \mathbb{N}}$ that converges to some $u\in \mathbb{S}$ weakly in $H^1$  and strongly in $L^{p}$ (with $p<6$ for $d\le 3$). Using \eqref{Energy-diff-by-H1-diff} we have $\lim\limits_{n\rightarrow \infty} \| \hat{u}^{n+1} - u^n\|_{H^1(\D)}=0$ and hence 
$$
1 = \lim\limits_{n\rightarrow \infty} \| \hat{u}^{n+1} \|_{L^2(\D)}
=  \lim\limits_{n\rightarrow \infty} \|  (1-\tau_n)u^n + \tau_n \gamma^n \Acalinv{u^n} u^n \|_{L^2(\D)}.
$$
Next, we use the coercivity of $\Acal_{|u|}\mathcal{I}$ % 
together with $\| u^{n_j} - u\|_{L^{2}(\D)} \rightarrow 0$, to conclude that 
$$
\| \Acalinv{u} u^{n_j}  -  \Acalinv{u} u\|_{H^1(\D)} \le C \|  u^{n_j} - u \|_{L^2(\D)} \rightarrow 0,
$$
i.e., $\Acalinv{u} u^{n_j} \rightarrow \Acalinv{u} u$ in $H^1(\D)$. To ensure convergence of $\Acalinv{u^{n_j}} u^{n_j}$ to $\Acalinv{u} u$ (weakly in $H^1$ and strongly in $L^2$), it remains to check $\Acalinv{u} u^{n_j} - \Acalinv{u^{n_j}} u^{n_j}$. For this, we use the convergence  $\| u^{n_j} - u\|_{L^{4}(\D)} \rightarrow 0$ together with the estimate $\| \Acalinv{u^{n_j}} u^{n_j} \|_{H^1(\D)} \le C \|  u^{n_j} \|_{L^2(\D)} = C$ to obtain for any $v\in H^1_0(\D)$ that
\begin{eqnarray*}
\lefteqn{ | \langle \Acal_{|u|} (\Acalinv{u} u^{n_j} - \Acalinv{u^{n_j}} u^{n_j}) ,v  \rangle |
\,\,= \,\, \beta | ( \,(| u^{n_j}|^2 -|u|^2) \, \Acalinv{u^{n_j}} u^{n_j} , v )_{L^2(\D)} | } \\
&\le& \beta\, \| \, | u^{n_j}|^2 -|u|^2 \|_{L^2(\D)} \| \Acalinv{u^{n_j}} u^{n_j} \|_{L^4(\D)} \| v\|_{L^4(\D)}  \\
&\le& C \, \beta\, \| \, | u^{n_j}|^2 -|u|^2 \|_{L^2(\D)} \| v\|_{H^1(\D)} 
\,\, \le \,\, C \, \beta\, \| \, u^{n_j} - u \|_{L^4(\D)} \| v\|_{H^1(\D)}   \rightarrow 0
\end{eqnarray*}
for $n_j \rightarrow \infty$ and a constant $C$ depending on $\Omega$, $d$ and on the uniform $H^1$-bounds for $u^n$. Since $\langle \Acal_{|u|} \, \cdot,\cdot\rangle$ is equivalent to the $H^1$-inner product on $H^1_0(\D)$, we conclude altogether that $\Acalinv{u^{n_j}} u^{n_j} \rightharpoonup \Acalinv{u} u$ in $H^1$ and $\Acalinv{u^{n_j}} u^{n_j} \rightarrow \Acalinv{u} u$ in $L^2$.
Therefore, $\gamma^{n_j} = ( u^{n_j} , \Acalinv{u^{n_j}} u^{n_j} )_{\0}^{-1}$ must converge to $\lambda := ( u ,\Acalinv{u} u  )_{\0}^{-1}$. We are ready to combine the above identified limits to obtain for arbitrary $v\in \3$:
\begin{eqnarray*}
\lefteqn{ 0 \overset{j \rightarrow \infty} \longleftarrow \hspace{3pt}
 \tau_{n_j}^{-1} \langle \Acal_{|u|} ( \hat{u}^{n_{j} +1} - u^{n_j} ), v \rangle }\\
&=&   \langle \Acal_{|u|}( - u^{n_j} + \gamma^{n_j} \Acalinv{u^{n_j}} u^{n_j}, v \rangle \overset{j \rightarrow \infty}\longrightarrow -  \langle \Acal_{|u|} u , v \rangle + \lambda \,  \langle \Acal_{|u|} \Acalinv{u} u, v \rangle.
\end{eqnarray*}
Thus, $E^{\prime}(u)= \Acal_{|u|} u = \lambda \mathcal{I}u$, which proves the first part of (iii). It remains to verify the strong $H^1$-convergence. Here we obtain with $\gamma^{n_j} = \tau_{n_j}^{-1} \langle \Acal_{|u^{n_j}|}(\hat{u}^{n_j +1} - u^{n_j}) , u^{n_j} \rangle + \langle \Acal_{|u^{n_j}|}  u^{n_j} , u^{n_j} \rangle$ that
\begin{eqnarray*}
\lefteqn{2 \hspace{2pt} | E(u) - E(u^{n_j})|  \,\,\,=\,\,\, \left| \lambda - \gamma^{n_j} + \gamma^{n_j} - \tfrac{\beta}{2} \int_{\D} |u|^4 \dx
- \langle \Acal_{|u^{n_j}|} u^{n_j},u^{n_j} \rangle + \tfrac{\beta}{2} \int_{\D} |u^{n_j}|^4 \dx\right| }\\
& \overset {\eqref{pseudo-energy-identity}}\le& | \lambda - \gamma^{n_j}| + \tfrac{\beta}{2} \int_{\D}\bigl| |u^{n_j}|^4 - |u|^4  \bigr|\dx  +  \tau_{n_j}^{-2} \,  \langle \Acal_{|u^{n_j}|}(\hat{u}^{n_j +1} - u^{n_j}) , \hat{u}^{n_j +1} - u^{n_j} \rangle \\
&\overset{\eqref{Energy-diff-by-H1-diff}}{\le} &  | \lambda - \gamma^{n_j}| + \tfrac{\beta}{2} \int_{\D} \bigl| |u^{n_j}|^4 - |u|^4  \bigr|\dx + c(\tau_{\mbox{\tiny\rm min}}, \tau_{\mbox{\tiny\rm max}}) \, \big( E( u^{n_j }) - E(u^{n_{j}+1}) \big) \overset{j \rightarrow \infty}{\longrightarrow} 0
\end{eqnarray*}
Passing to the limit on both sides yields $E(u) - \boldsymbol{E}=\lim\limits_{j\rightarrow \infty} E(u) - E(u^{n_j})=0$, which implies, together with the weak convergence in $H^1$, strong convergence in $H^1$. This proves (iii). 
It remains to check (iv). Since the limit energy $\mathbf{E}$ is unique, all limits $u \in \mathbb{S}$ of subsequences in (iii) have the same energy level. Now assume that one such limit is a locally quasi-unique ground state, call it $u$, then all accumulation points $\hat{u} \in \mathbb{S}$ of $\{ u^n \}_{n \in \mathbb{N} }$ must be ground states. To account for the phase shifts, we introduce the equivalence class $[v]:=\{ \exp(\ci \phase) v \, |\,\phase \in [-\pi,\pi)\}$ and define the distance between $[v]$ and $[w]$ by
\begin{align*}
\mbox{dist}( [v] , [w] ) := \inf_{(v,w) \in [v] \times [w] } \| v - w\|_{H^1(\D)}. 
\end{align*}
Now consider the set $\mathcal{Z}$ of equivalence classes of accumulation points of $\{ u^n \}_{n \in \mathbb{N} }$, 
i.e., 
$$
\Zu :=\{ \,[\hat{u}]  \, | \, \hat{u} \in \mathbb{S} \, \text{ is an accumulation point of } \, \{u^n\}_{n\in \mathbb{N}}
 \, \}.
$$
Since $u$ is locally quasi-unique and $\Zu$ only contains equivalence classes of ground states, $\mathcal{Z} \setminus [u]$ is closed and we have
\begin{align*}
2\delta := \mbox{dist}(\,[u]\hspace{1pt},  \Zu \hspace{-1pt}\setminus \hspace{-1pt} [u] ) := \inf_{[v] \in \Zu \setminus [u] } \mbox{dist}( [u] , [v] ) > 0.
\end{align*}
Therefore, there is no $[\hat{u}] \in\Zu \hspace{-1pt}\setminus \hspace{-1pt} [u] $ which satisfies that $\delta < \mbox{dist}([u] ,[\hat{u}] ) < \delta + \eps$ for $\eps\in (0, \delta)$.
Hence, with \,$\mathbb{B}_{\delta}([u]) \hspace{-1pt}:= \hspace{-1pt}\{ \, [v] \,|\,\, v\in \mathbb{S},\,\, \mbox{dist}( [u] , [v] ) < \delta\,\}$\, the sequence $\{ [u^n] \}_{n\in\mathbb{N}}$ 
has only finitely many elements that satisfy 
\begin{equation}
\label{finite}
 \delta  < 
  \mbox{dist}([u] ,[u^n] ) <  \delta  + \eps 
\end{equation}
(otherwise we could select a weakly converging subsequence among these elements $u^n$ and repeat the same arguments as in (iii) to prove that this subsequence is strongly converging in $H^1$ to a ground state, whose equivalence class would be, by definition, an element of $\Zu \setminus [u] $ and hence leading to a contradiction that the distance of $\Zu \setminus [u] $ to $[u]$ is $2\delta$).
On the other hand, we have $\|u^{n+1} - u^n \|_{\1} \rightarrow 0$ for $n \rightarrow \infty$ (cf. Corollary \ref{goes_to_zero}), which implies $\mbox{dist}([u^{n+1}],[u^n]) \rightarrow 0$. Consequently, by selecting $N \in \mathbb{N}$ large enough so that $[u^{n}] \not\in \mathbb{B}_{\delta + \eps }([u]) \setminus \mathbb{B}_{\delta}([u])$ for all $n\ge N$, we conclude from  $\mbox{dist}([u^{n+1}],[u^n]) \rightarrow 0$ that $[u^n] \in \mathbb{B}_{\delta}([u])$ for all $n\ge N$. Hence, $[u^n]$ converges to $[u]$, which proves (iv) as, with the diamagnetic inequality,
\begin{align*}
\| \, |u^n| - |u| \, \|_{H^1(\D)} \le 
\inf_{ \phase_1,\phase_2 \in [-\pi,\pi) } \| \, \exp(\ci \phase_1) u^n - \exp(\ci \phase_2) u \|_{H^1(\D)}
= \mbox{dist}([u^{n}],[u]) \rightarrow 0.
\end{align*}
\end{proof}
\section{Proof of Theorem \ref{linear_rate_theorem}: Local convergence results}
\label{section-local-convergence-proof}
We will now turn to the local convergence analysis and the quantification of corresponding rates in the case of a fixed step size $\tau=\tau_n$ for all $n$. For that, note that the iterations \eqref{method} can be seen as fixed-point iteration $u^{n+1}=\phi_{\tau}(u^n)$ with fixed-point function $\phi_{\tau}: \3 \rightarrow \3$ given by
\begin{align}
\label{fixed-point-scheme}
\phi_{\tau} ( v ):= \frac{ \hspace{-4pt}\psi_{\tau}(v)}{\| \psi_{\tau}(v) \|_{\0}\hspace{-20pt}}\qquad\hspace{10pt} \mbox{where } \qquad
\psi_{\tau}(v):= (1-\tau) v + \tau \, \gamma(v) \, \Acalinv{v} v \, ,
\end{align}
and $\gamma(v):= (\Acalinv{v} v , v)_{\0}^{-1}$. We want to apply Ostrowski's theorem in Banach spaces (cf. \cite{AHP21NumMath,Shi81}) to guarantee local convergence with the convergence rate given by the spectral radius of the differential $\phi_{\tau}^{\prime}(u)$, provided that this spectral radius is strictly smaller than $1$. However, Ostrowski's theorem is not directly applicable to the above iteration, as we would find the spectral radius to become $1$ for the eigenfunction $\ci u$ which corresponds to complex phase shifts. In fact, if $u$ is a fixed point of $\phi_{\tau}$, then $\exp(\ci \phase) u$ is another fixed point (for any $\phase \in [-\pi,\pi)$) that can be arbitrarily close. Therefore, we need to introduce an auxiliary iteration which artificially adjusts the phase in each iteration. After establishing local convergence for the auxiliary scheme, the results can be transferred to the original scheme \eqref{method}. 
%
%%%%%%%%%%%%
\begin{definition}[Auxiliary iteration]
\label{auxi-scheme-definition} 
Let $u \in \mathbb{S}$ be a locally quasi-unique ground state. For a given initial value $\tilde{u}^0 \in \mathbb{S}$, the auxiliary iterates $\tilde{u}^{n+1} \in \3$  for $n \geq 0$ are defined by
\begin{align}
\label{auxi-scheme}
\tilde{u}^{n+1} =   \tilde{\phi}_{\tau}(\tilde{u}^n) ,
\end{align} 
where the fixed point operator  $\tilde{\phi}_{\tau}: D(\tilde{\phi}_{\tau}) \subset \3 \rightarrow \3$ is given by
\begin{align}
\label{def-tilde-phi-tau-v}
\tilde{\phi}_{\tau}(v):= \Theta_{u}(v) \, \phi_{\tau}(v) \qquad \mbox{for } \, v \in D(\tilde{\phi}_{\tau}),
\end{align}
with  $\Theta_{u}(v) := \tfrac{\theta_{u}(v)}{| \theta_{u}(v) |} \, ,\, \theta_{u}(v) :=\overline{\int_{\D} \psi_{\tau}(v) \, \overline{u} \dx }$ and $\psi_{\tau}(v) , \phi_{\tau}(v)$ are as defined in \eqref{fixed-point-scheme}. The iterates are well-defined in a neighborhood of $u$ because it holds $\Theta_{u}(u)=1$.
\end{definition}
To better understand the iterations \eqref{auxi-scheme}, note that $u$ is still a fixed point, i.e., $\tilde{\phi}_{\tau}(u)=u$, but phase shifts of $u$ are no longer fixed points as
\begin{align*}
\tilde{\phi}_{\tau}(\exp(\ci \phase) u) &=  \tfrac{\theta_{u}( \exp(\ci \phase) u )}{| \theta_{u}( \exp(\ci \phase) u ) |} \, \phi_{\tau}( \exp(\ci \phase) u ) =  \exp(- \ci \phase) (\exp( \ci \phase)u) \\
&= u \not= \exp( \ci \phase)u 
\qquad \mbox{for any } \phase \in [-\pi , 0) \cup (0, \pi ).
\end{align*}
Furthermore, if initialized with the same starting value, the iterates $\tilde{u}^n$ and $u^n$ only deviate in a constant phase factor. This is summarized in the following lemma.
\begin{lemma}
\label{lemma-relation-un-tilde-un}
Assume \ref{A1}-\ref{A4} and let $u\in \mathbb{S}$ denote a ground state. Consider the iterates $u^n$ of the Riemannian gradient method \eqref{method} with constant step size $\tau$ and let $\tilde{u}^n$ denote the companion auxiliary iterates generated by \eqref{auxi-scheme}. If both iterations are initialized with the same starting value $\tilde{u}^0 = u^0 \in \mathbb{S}$ and if $\theta_{u}(u^n) \not= 0$ for all $n\ge 0$, then it holds
\begin{align}
\label{relation-un-tilde-un}
\tilde{u}^{n+1} = \Theta_{u}(u^n) \, u^{n+1}.
\end{align}
\end{lemma}
The above lemma will be later used to prove local convergence for $u^n$.
\begin{proof}
Let $\phase \in [-\pi ,\pi)$ be arbitrary and let us for brevity define the linear gauge transform $G_{\phase}v:=\exp(\ci \phase) v$. For any $v,w \in H^1_0(\D)$ it holds
\begin{align*}
\langle \Acal_{|v|}( \Acalinv{v}  v ), w \rangle &= (v,w)_{\0}= ( G_{\phase}v , G_{\phase}w )_{\0} 
= \langle \Acal_{|G_{\phase}v|}( \Acalinv{G_{\phase}v}  (G_{\phase}v)) , G_{\phase} w \rangle \\
&= \langle \Acal_{|v|}( \Acalinv{G_{\phase}v}  (G_{\phase}v)) , G_{\phase} w \rangle 
= \langle \Acal_{|v|}( G_{-\phase} \Acalinv{G_{\phase}v}  (G_{\phase}v)) , w \rangle 
\end{align*}
and therefore $\Acalinv{G_{\phase}v}(G_{\phase}v) = G_{\phase}(\Acalinv{v}v)$. Consequently, we have
 \begin{eqnarray*} 
\psi_{\tau}(G_{\phase}v) &=&   (1-\tau) G_{\phase}v + \tau ( \Acalinv{G_{\phase}v}(G_{\phase}v) ,G_{\phase}v )_{L^2(\D)}^{-1} \Acalinv{G_{\phase}v}(G_{\phase}v) \\
&=&  (1-\tau) G_{\phase}v + \tau ( G_{\phase}(\Acalinv{v}v) ,G_{\phase}v )_{L^2(\D)}^{-1} G_{\phase}(\Acalinv{v}v) \\
&\overset{|\exp(\ci \phase)|=1}{=}&  G_{\phase}\left( \, (1-\tau) v + \tau ( \Acalinv{v}v , v )_{L^2(\D)}^{-1} \Acalinv{v}v \, \right) \,\,\,=\,\,\,  G_{\phase}(\,\psi_{\tau}(v)\,).
\end{eqnarray*}
We conclude
\begin{align}
\label{gauge-transform-in-iteration}
\phi_{\tau} ( \, G_{\phase}v \,  )= \frac{ \hspace{-4pt}\psi_{\tau}(G_{\phase}v)}{\| \psi_{\tau}(G_{\phase}v) \|_{\0}\hspace{-20pt}}\hspace{15pt}
= \frac{ \hspace{-4pt}G_{\phase}(\,\psi_{\tau}(v)\,)}{\| \psi_{\tau}(v) \|_{\0}}
=  G_{\phase}(\, \phi_{\tau} ( v  )\,).
\end{align}
The claim \eqref{relation-un-tilde-un} can be now proved by induction. The base step $n=0$ follows readily with
\begin{align*}
\tilde{u}^1 = \tilde{\phi}_{\tau}(\tilde{u}^0) &= \Theta_{u}(\tilde{u}^0) \phi_{\tau}(\tilde{u}^0) \overset{\tilde{u}^0 = u^0}=  \Theta_{u}(u^0)  \,  \phi_{\tau}(u^0) =  \Theta_{u}(u^0) \, u^1 .
\end{align*}
For the induction step $n\mapsto n+1$, let $\tilde{u}^{n} = \Theta_{u}(u^{n-1}) \, u^n$ hold. Furthermore, note that since $\theta_{u}(u^{n-1}) \not= 0$, we have $|\Theta_{u}(u^{n-1})|=1$ and hence there exists a $\phase_{n-1} \in [-\pi,\pi)$ such that  $\Theta_{u}(u^{n-1}) = \exp( \ci \phase_{n-1})$. Hence, with $\tilde{u}^n = G_{\omega_{n-1}} u^n$ we obtain
\begin{align*}
\tilde{u}^{n+1} &= \tilde{\phi}_{\tau}(\tilde{u}^{n}) 
= \tilde{\phi}_{\tau}(G_{\omega_{n-1}} u^n) = \Theta_{u}(G_{\omega_{n-1}} u^n) \,\phi_{\tau}(G_{\omega_{n-1}} u^n) \\
&=  \tfrac{ \overline{\int_{\D} \psi_{\tau}(G_{\omega_{n-1}} u^n) \, \overline{u}\, \mbox{\scriptsize d}x } }{| \overline{\int_{\D} \psi_{\tau}(G_{\omega_{n-1}} u^n) \, \overline{u} \, \mbox{\scriptsize d}x }  |} \,\phi_{\tau}(G_{\omega_{n-1}}u^n) 
\overset{\eqref{gauge-transform-in-iteration}}{=}  G_{-\omega_{n-1}} \hspace{-2pt} \left( \tfrac{ \overline{\int_{\D} \psi_{\tau}(u^n) \, \overline{u}\, \mbox{\scriptsize d}x } }{| \overline{\int_{\D} \psi_{\tau}(u^n) \, \overline{u} \, \mbox{\scriptsize d}x }  |} \right)\, G_{\omega_{n-1}}(\,\phi_{\tau}(u^n)\,) 
\\
&=  \tfrac{ \overline{\int_{\D} \psi_{\tau}(u^n) \, \overline{u}\, \mbox{\scriptsize d}x } }{| \overline{\int_{\D} \psi_{\tau}(u^n) \, \overline{u} \, \mbox{\scriptsize d}x }  |} \, \phi_{\tau}(u^n)    = \Theta_{u}(u^n) \, \phi_{\tau}(u^n) = \Theta_{u}(u^n) \, u^{n+1},
\end{align*}
where we used the properties $G_{-\omega}(v)\,G_{\omega}(w)=v\,w$ and $|G_{\omega}(v)|=|v|$. This proves \eqref{relation-un-tilde-un}.
\end{proof}
In order to establish local convergence of the auxiliary iterations \eqref{auxi-scheme} with Ostrowski's theorem, we need to characterize the spectrum of $\tilde{\phi}^{\prime}_{\tau}(u)$ and bound its spectral radius afterwards. The first step is taken in the next subsection.

\subsection{The operator $\tilde{\phi}^{\prime}_{\tau}(u)$}
This subsection is devoted to the calculation of the Fr\'echet derivative of the (auxiliary) fixed-point function $\tilde{\phi}_{\tau}$ at $u$ and various technical results regarding its spectrum. First, we note that $\tilde{\phi}_{\tau}$ is Fr\'echet differentiable in an open neighborhood $U\subset H^1_0(\D)$ around $u$. For any $v \in U$ and direction $h \in \3$, the corresponding derivative is given by the product rule as
\begin{align}
\label{simple-product-rule}
\tilde{\phi}_{\tau}'(v)h= \Theta_{u}(v) \, \phi_{\tau}'(v)h + \phi_{\tau}(v) \, \Theta_{u}'(v)h.
\end{align}
The derivatives of $\phi_{\tau}$ and $\Theta_{u}$ at $u$ are explicitly derived in the following lemmas. We begin with $\phi_{\tau} $ for which we first need to calculate the derivatives of $\Acalinv{v} v$ and $\gamma(v)$ at $v =u$.
\begin{lemma}
\label{psi-derivative_lemma}
Assume \ref{A1}-\ref{A4}. The map $\psi : \3 \rightarrow \3$ with $\psi(v):=\Acalinv{v} v$ is Fr\'echet differentiable for all $v\in \3$ with derivative $\psi^{\prime}(v): \3 \rightarrow \3$ given by
\begin{align}
\label{psi-derivative}
\psi^{\prime}(v) h = \Acalinv{v}(h - 2 \beta \,\re(v \overline{h}) \, \psi(v)  )
\end{align}  
for $h\in \3$. In particular, we have for $v=u$
\begin{align}
\label{psi-derivative-ground-state}
\psi^{\prime}(u) h =  \Acalinv{u} (h - \tfrac{2 \beta}{\lm} \,\re(u \overline{h})u ).
\end{align}
\end{lemma}
\begin{proof}
Let $v\in \3$ be a position and $h\in \3$ a direction. With $\psi(v)=\Acalinv{v} v$ we get
\begin{align}
\label{deriv-psi-prod-rule}
\mathcal{I}h &= 
 (\tfrac{\mbox{\scriptsize d}}{\mbox{\scriptsize d} v } \langle \mathcal{A}_{|v|} \psi(v) , \cdot \,\rangle ) \,h 
= 
2 \beta \, \mathcal{I} ( \re(v \overline{h}) \psi(v) )
+ \Acal_{|v|} (\psi^{\prime}(v) h).
\end{align}
where the last equality is an application of the product rule together with the observation that $\tfrac{\mbox{\scriptsize d}}{\mbox{\scriptsize d} v } \langle \mathcal{A}_{|v|}w , \cdot \,\rangle\,h = 2 \beta \, \mathcal{I} ( \re(v \overline{h}) w )$. 
From \eqref{deriv-psi-prod-rule} we obtain that $ \psi^{\prime}(v) h \in \3$ fulfills 
\begin{align}
\label{proof-psi-derivative}
\Acal_{|v|}(\psi^{\prime}(v) h ) = \mathcal{I}(h - 2 \beta \,\re(v \overline{h}) \, \psi(v)  ).
\end{align}
From \eqref{proof-psi-derivative} we conclude \eqref{psi-derivative}, which is well-defined because $(h - 2 \beta \,\re(v \overline{h}) \, \psi(v))  \in L^2(\D)$. The latter follows with the H\"older inequality and the Sobolev embedding of $\3$ into $L^6(\D)$. In particular, $\psi^{\prime}(v) : h \mapsto \Acalinv{v} (h - 2 \beta \,\re(v \overline{h}) \, \psi(v)  )$ is a continuous linear operator from $H^1_0(\D)$ to $H^1_0(\D)$, hence, the Fr\'echet derivative of $\psi(v)$ exists. Equality \eqref{psi-derivative-ground-state} now follows from \eqref{psi-derivative} with the choice $v=u$ and accordingly $\psi(u)=\Acalinv{u}u=\lm^{-1} u$.
\end{proof}
Next, we give expressions for $\gamma'(u)$ and $\psi_{\tau}^{\prime}(u)$.
\begin{lemma}
Assume \ref{A1}-\ref{A4}, then 
$\gamma(v)\hspace{-1pt}=\hspace{-1pt} (\Acalinv{v} v , v)_{\0}^{-1}$ and $\psi_{\tau}(v)\hspace{-1pt}=\hspace{-1pt} (1\hspace{-1pt}-\hspace{-1pt}\tau) v \hspace{-1pt}+\hspace{-1pt} \tau \gamma(v) \psi(v)$ are Fr\'echet-differentiable on $H^1_0(\D) \setminus \{ 0 \}$. If $v=u$ is a ground state it holds for $h \in \3$
\begin{align*}
 \gamma^{\prime}(u) h =   - 2 \, \lambda \, \langle \mathcal{I}( h - \tfrac{\beta}{\lambda} \, \re(u \overline{h}) \, u ) , u \rangle 
\end{align*}
and
\begin{align}
\label{psi-tau-u-derivative}
 \psi_{\tau}^{\prime}(u) h  =(1-\tau)h - 2\, \tau \, \langle \mathcal{I} ( \hspace{1pt} h -  \tfrac{\beta}{\lambda} \, \re(u \overline{h}) u )\hspace{1pt} , u \rangle \, u +  \tau \, \lambda \,  \Acalinv{u} \hspace{1pt} (h - 2 \tfrac{\beta}{\lambda} \, \re(u \overline{h}) u \hspace{1pt} ).
\end{align} 
\end{lemma}
\begin{proof}
For $v\in \3 \setminus \{ 0\}$, the Fr\'echet derivative of $\gamma(v)$ in direction $h\in \3$ is
\begin{align}
\label{gamma-1-eqn}
 \gamma^{\prime}(v) h  =  - \gamma(v)^2 \big( \gamma(v)^{-1} \big)' (h)=  - \gamma(v)^2 \left(  ( \psi(v), h )_{\0} + ( \, \psi'(v) h \, , v )_{\0}  \right), 
\end{align}
where we recall $\psi(v)= \Acalinv{v} v$. 
 Evaluating the expression \eqref{gamma-1-eqn} for a ground state $u$, together with $\psi(u) = \lm^{-1} u$, $\gamma(u) = \lambda$ and \eqref{psi-derivative-ground-state}, we have
\begin{align*}
 \gamma^{\prime}(u) h  &=  - \lambda^2 \left( (\Acalinv{u}( \hspace{1pt} h - 2 \tfrac{\beta}{\lambda} \, \re(u \overline{h})u \hspace{1pt} ) , u  )_{\0} +  \lambda^{-1}  ( u, h )_{\0} \right) \\
 & = -\lm \left( \langle \Acal_{|u|} \Acalinv{u} (h - 2 \tfrac{\beta}{\lambda} \, \re(u \overline{h}) u) , u \rangle + \langle \mathcal{I} h , u \rangle \right) 
 \,=   - 2 \, \lambda \,  \langle \mathcal{I} ( \hspace{1pt} h -  \tfrac{\beta}{\lambda} \, \re(u \overline{h}) u) \hspace{1pt} , u \rangle.
\end{align*}
With $\psi_{\tau}(v)= (1-\tau) v + \tau \gamma(v) \psi(v)$ and $\psi(v)=\Acalinv{v} v$, we further obtain
\begin{align*}
 \psi_{\tau}^{\prime}(u) h &= (1-\tau)h + \tau \, (\gamma^{\prime}(u) h) \, \Acalinv{u} u +  \tau \, \gamma(u) \,  \psi^{\prime}(u) h  \\
&= (1-\tau)h - 2\, \tau \,  \langle \mathcal{I} ( \hspace{1pt} h -  \tfrac{\beta}{\lambda} \, \re(u \overline{h}) u) \hspace{1pt} , u \rangle \, u +  \tau \, \lambda \, \psi^{\prime}(u) h \\
&\overset{\eqref{psi-derivative-ground-state}}{=} (1-\tau)h - 2\, \tau \,  \langle \mathcal{I} ( \hspace{1pt} h -  \tfrac{\beta}{\lambda} \, \re(u \overline{h}) u) \hspace{1pt} , u \rangle \, u +  \tau \, \lambda \,  \Acalinv{u} ( \hspace{1pt} (h - 2 \tfrac{\beta}{\lambda} \, \re(u \overline{h}) u \hspace{1pt} ).\\[-3.0em]
\end{align*} 
\end{proof}
Since we want to compute the derivative of $\tilde{\phi}_{\tau}(v)= \Theta_{u}(v) \, \phi_{\tau}(v) = \tfrac{ \Theta_{u}(v)\psi_{\tau}(v)}{\| \psi_{\tau}(v) \|_{L^2(\D)}}$, we need to study two more components, that is, $ \phi_{\tau}^{\prime}(v)$ and $\Theta_u^{\prime}(v)$. For $\phi_{\tau}$ we get straightforwardly for any $v\in H^1_0(\D)$ with $\psi_{\tau}(v) \not=0$ and $h \in \3$ 
\begin{align*}
\phi_{\tau}'(v)h = \frac{\psi_{\tau}'(v)h}{\| \psi_{\tau}(v)\|_{\0}} - \frac{(\psi_{\tau}'(v)h , \psi_{\tau}(v))_{\0} \psi_{\tau}(v)}{\|\psi_{\tau}(v)\|^3_{\0}} .
\end{align*}
In particular, for $v=u$ we obtain with $\psi_{\tau}(u) = u$
\begin{align}
\label{phi-derivative-final}
\phi_{\tau}'(u)h = \psi_{\tau}'(u)h - (\psi_{\tau}'(u)h , u)_{\0} u.
\end{align} 
Here, $\psi'_{\tau}(u)$ is given by \eqref{psi-tau-u-derivative}. It remains to give an expression for $\Theta_{u}^{\prime}$ in $u$.
\begin{lemma} \label{Theta-derivative}  Assume \ref{A1}-\ref{A4}. For a ground state $u$, the Fr\'echet derivative of $\Theta_{u}$ at $u$, for any $h \in \3$ is given by
\begin{align*}
\Theta'_{u}(u)h=\ci \, \im(\theta_{u}'(u)h)
\qquad
\mbox{and}
\qquad
 \theta_{u}'(u)h = \overline{ \int_{\D} \psi_{\tau}'(u)h \, \overline{u} \dx }
\end{align*}
\end{lemma}
\begin{proof} 
For $v\in H^1_0(\D)$ from a sufficiently small neighborhood of $u$ such that $\theta_{u}(v)\not=0$, we get 
\begin{align*}
\Theta_{u}'(v)h = \frac{\theta_{u}'(v)h}{|\theta_{u}(v)|} - \frac{\re\big(\theta_{u}'(v)h \,\, \overline{\theta_{u}(v)} \, \big) \theta_{u}(v)}{|\theta_{u}(v)|^3} \qquad \mbox{with} \qquad \theta_{u}'(v)h = \overline{ \int_{\D} \psi_{\tau}'(v)h \, \overline{u} \dx} .
\end{align*}
Note that by definition, we have $\theta_{u}(u) =  \overline{\int_{\D} \psi_{\tau}(u) \overline{u} \dx} = \overline{\int_{\D} u \overline{u} \dx} = 1.$ This implies 
\begin{align*}
\Theta'_{u}(u)h= \theta_{u}'(u)h - \re\big(\theta_{u}'(u)h \big) = \ci \, \im(\theta_{u}'(u)h).\\[-3.7em]
\end{align*}
\end{proof}
By collecting the previous results, in particular \eqref{phi-derivative-final} and Lemma \ref{Theta-derivative}, we obtain for the derivative of $\tilde{\phi}_{\tau}(v)= \Theta_{u}(v) \, \phi_{\tau}(v)$ in $u$ that
\begin{align} 
\nonumber  \tilde{\phi}_{\tau}'(u)h &= \Theta_{u}(u) \, \phi_{\tau}'(u)h + \phi_{\tau}(u) \, \Theta_{u}'(u)h = \phi_{\tau}'(u)h + u \, \Theta_{u}'(u)h \\
\label{tilde-phi-derivative-2} &= \psi_{\tau}'(u)h - (\psi_{\tau}'(u)h , u)_{\0} u + \im(\theta'_{u}(u)h) (\ci u) .
\end{align}
From \eqref{tilde-phi-derivative-2} we can see that the problematic eigenvalue $1$ of $\phi_{\tau}'(u)$ to the eigenfunction $\ci u$, is converted (by phase locking through $\Theta_u$) to the eigenvalue $0$ of $\tilde{\phi}_{\tau}^{\prime}(u)$. To see this, first note that $\psi_{\tau}'(u)(\ci u) = \ci u$ which implies, with Lemma \ref{Theta-derivative}, $\theta'_{u}(u)(\ci u)=-\ci$. Exploiting this yields
\begin{align*}
\phi_{\tau}'(u)(\ci u) & \overset{\eqref{phi-derivative-final}} = \psi_{\tau}'(u)(\ci u) - (\psi_{\tau}'(u)(\ci u) , u)_{\0} u = \ci u \, , \qquad \mbox{whereas}\\
\tilde{\phi}_{\tau}'(u) (\ci u) & \overset{\eqref{tilde-phi-derivative-2}}= \psi_{\tau}'(u)(\ci u) - (\psi_{\tau}'(u)(\ci u) , u)_{\0} u + \im(\theta'_{u}(u)(\ci u)) (\ci u)  = \ci u - \ci u =0.
\end{align*}
To simplify the representation of $\tilde{\phi}_{\tau}'(u)$ and its spectrum, we need an additional orthogonality result, stating that all eigenfunctions of $\tilde{\phi}'_{\tau}(u)$ that are not in the kernel are $L^2$-orthogonal to both $u$ and $\ci u$. Note that eigenfunctions in the kernel are not relevant for the spectral radius.
\begin{lemma} \label{eigenvector-orthogonal}
Assume \ref{A1}-\ref{A4} and let $u$ be a ground state. For $\tilde{\phi}'_{\tau}(u):\3 \rightarrow \3$, we consider the eigenvalue problem seeking $v_i \in \3$ and $\tilde{\mu}_i \in \mathbb{R}$ such that 
\begin{align}
\label{ev-problem-tilde-phi-prime}
\tilde{\phi}_{\tau}'(u) v_i = \tilde{\mu}_{i} v_i .
\end{align}
Then, for all eigenvalues $\tilde{\mu}_i \neq 0$ with eigenfunctions $v_i \in \3$ it holds $v_i \in \5$.
\end{lemma}
\begin{proof}
The eigenvalue problem \eqref{ev-problem-tilde-phi-prime} can be written as
\begin{align*}
( \tilde{\phi}_{\tau}'(u)v_i , w)_{\0} =\tilde{\mu}_{i}( v_i, w)_{\0} \qquad \mbox{ for all} \quad w \in \3.
\end{align*}
$\bullet$ {\it Orthogonality w.r.t. $u$:} Selecting $w= u$ and using $\im(\theta'_{u}(u)v_i) \in \mathbb{R}$ and $u \in \mathbb{S}$ we get
\begin{align*}
\tilde{\mu}_{i}( v_i, u)_{\0} &\overset{\eqref{tilde-phi-derivative-2}}= ( \psi_{\tau}'(u)v_i , u)_{\0} - (\psi_{\tau}'(u)v_i , u)_{\0} + \im(\theta'_{u}(u)v_i) (\ci u , u)_{\0} =0.
\end{align*}
\hspace{10pt}This implies $( v_i, u)_{\0}=0$ if $\mu_i \not=0$.\\[0.5em]
$\bullet$ {\it Orthogonality w.r.t. $\ci u$:} Selecting $w= \ci u$ we get
\begin{eqnarray*}
\lefteqn{\tilde{\mu}_{i}( v_i, \ci u)_{\0} 
 \,\overset{\eqref{tilde-phi-derivative-2}}=\, ( \psi_{\tau}'(u)v_i , \ci u)_{\0} \hspace{-1pt}-\hspace{-1pt} (\psi_{\tau}'(u)v_i , u)_{\0}(u, \ci u)_{\0} \hspace{-1pt}+\hspace{-1pt} \im(\theta'_{u}(u)v_i) (\ci u , \ci u)_{\0} } \\
&=&( \psi_{\tau}'(u)v_i , \ci u)_{\0} + \im(\theta'_{u}(u)v_i) (\ci u , \ci u)_{\0} \overset{\eqref{imthetaprimeu}}{=}0 \,,\hspace{170pt}
\end{eqnarray*}
\hspace{8pt}where we have used that 
\begin{eqnarray}
\label{imthetaprimeu}\lefteqn{\im(\theta'_{u}(u)v_i) = \im\big( \overline{ \int_{\D} \psi'_{\tau}(u)v_i \, \overline{u} \dx} \big)= - \im\big( \int_{\D} \psi'_{\tau}(u)v_i \, \overline{u} \dx \big) }\\
\nonumber&=& \re\big( \ci \int_{\D} \psi'_{\tau}(u)v_i \, \overline{ u} \dx \big) = -\re\big( \int_{\D} \psi'_{\tau}(u)v_i \, \overline{\ci u} \dx \big) = -( \psi_{\tau}'(u)v_i , \ci u)_{\0}.\\[-2.5em]
\nonumber
\end{eqnarray}
\end{proof}
With the collection of lemmas above, we are now ready to prove Theorem \ref{linear_rate_theorem}.
\subsection{Local convergence rates}
We start with quantifying the convergence properties of the auxiliary iteration $\tilde{u}^{n+1} =   \tilde{\phi}_{\tau}(\tilde{u}^n)$ given by \eqref{auxi-scheme}. For that we apply the Ostrowski theorem in Hilbert spaces which states, adapted to our setting, the following (see \cite{AHP21NumMath,Shi81} for a proof).
\begin{lemma}[Ostrowski]
\label{Ostrowski-result}
Assume \ref{A1}-\ref{A4}. Let $u \in \mathbb{S}$ be a ground state and consider the fixed-point function $\tilde{\phi}_{\tau}$ given by \eqref{def-tilde-phi-tau-v} with iterations $ \tilde{u}^{n+1}:=\tilde{\phi}_{\tau}(  \tilde{u}^{n} )$. If the spectral radius $\rho^{\ast}:=\rho(\,\tilde{\phi}_{\tau}^{\prime}(u)\,)$ of $\tilde{\phi}_{\tau}^{\prime}(u)$ fulfills $\rho^{\ast}<1$, then 
for every $\eps>0$ there is an open neighborhood $S_{\eps} \subset H^1_0(\D)$ of $u$ and a constant $C_{\eps}>0$ such that
$$ \| u - \tilde{u}^{n} \|_{H^1(\D)} \leq C_{\eps} \, | \rho^* + \eps |^n \, \| u - \tilde{u}^{0} \|_{H^1(\D)} $$
for all starting values $\tilde{u}^{0} \in S_{\eps}$ and $n\ge 1$. 
\end{lemma}
In order to apply Lemma \ref{Ostrowski-result}, we need to verify $\rho^{\ast}=\rho(\,\tilde{\phi}_{\tau}^{\prime}(u)\,)<1$, which is done in the following lemma.
\begin{lemma}
\label{lemma-estimates-rho-ast}
Assume \ref{A1}-\ref{A4}, let $u \in \mathbb{S}$ be a locally quasi-unique ground state and let $\tilde{\phi}_{\tau}$ by the fixed-point function from \eqref{def-tilde-phi-tau-v}. Then the spectral radius $\rho^{\ast}:=\rho(\,\tilde{\phi}_{\tau}^{\prime}(u)\,)$ is given by
\begin{align*}
\rho^{\ast} = \max_{i \in \mathbb{N} }\, | (1-\tau) + \tau \mu_i |,
\end{align*}
where $\mu_i$ denote the eigenvalues to \eqref{weighted-evp1}. If $\mu_1$ is the largest eigenvalue in magnitude, then
\begin{align}
\label{tau-conditions}
\rho^{\ast} < 1 \qquad \mbox{for all }  \tau \in 
\begin{cases}
(0,  1 +  \frac{\lambda_2 - \lambda_1}{\lambda_2+\lambda_1} ) 
&\mbox{if } \mu_1 > 0, \\
(0 , 1 + \frac{\delta_1}{2\lambda_1+\delta_1})
&\mbox{if } \mu_1 < 0,
\end{cases}
\end{align}
and where $\delta_1>0$ is as in Lemma \ref{lemma-weighted-eigenvalueproblem}.
\end{lemma}

\begin{proof}
First, note that the spectrum of $\tilde{\phi}'_{\tau}(u)$ can only contain the eigenvalues $\tilde{\mu}_i $ in \eqref{ev-problem-tilde-phi-prime} and the value $0$, cf. Appendix \ref{appendix-A}.
Since $0$ (respectively the eigenvalue $0$) is not relevant for determining the spectral radius, Lemma \ref{eigenvector-orthogonal} allows us to restrict the spectrum of  $\tilde{\phi}'_{\tau}(u)$ to $\5$, i.e., to the orthogonal complement of $u$ and $\ci u$. Note that for all $\tilde{\mu}_i \not =0$ we have $0 = \tilde{\mu}_i ( v_i , \alpha u + \beta \ci u )_{\0} = ( \tilde{\phi}^{\prime}_{\tau}(u)v_i , \alpha u + \beta \ci u )_{\0}$, for any $\alpha,\beta\in \R$. Hence, for $\tilde{\mu}_i \not=0$, the eigenvalue problem \eqref{ev-problem-tilde-phi-prime} reduces to: find $\tilde{\mu}_i \in \mathbb{R} \setminus \{0\}$ and $v_i \in \5$ with
\begin{align}
\label{EVp-tilde-phi}
 ( \tilde{\phi}^{\prime}_{\tau}(u)v_i , w )_{\0} =  \tilde{\mu}_i ( v_i , w )_{\0} \qquad \mbox{for all } w\in \5 .
\end{align}
When both $v_i$ and  $w$ belong to $\5$, the left hand side $( \tilde{\phi}^{\prime}_{\tau}(u)v_i , w )_{\0}$ simplifies to
\begin{eqnarray}
\nonumber \lefteqn{( \tilde{\phi}^{\prime}_{\tau}(u)v_i , w )_{\0} \,\overset{\eqref{tilde-phi-derivative-2}} =\, (\psi_{\tau}'(u)v_i , w)_{\0} \hspace{-2pt}-\hspace{-2pt} (\psi_{\tau}'(u)v_i , u)_{\0} (u,w)_{\0} \hspace{-2pt}+\hspace{-2pt} \im(\theta'_{u}(u)v_i) (\ci u,w)_{\0} }\\
\label{spectrum-of-tilde-phi}   &=& (\psi_{\tau}'(u)v_i , w)_{\0} \overset{\eqref{psi-tau-u-derivative}}{=}  (1-\tau) ( v_i , w )_{\0} + \tau \,\lambda \,( \Acalinv{u} ( \hspace{1pt} v_i -  2 \tfrac{\beta}{\lambda} \, \re(u \overline{v_i}) u \hspace{1pt} )  , w )_{\0}. \hspace{50pt}
\end{eqnarray}
Combining \eqref{EVp-tilde-phi} and \eqref{spectrum-of-tilde-phi} we see that $\tilde{\mu}_i$ is an eigenvalue of \eqref{EVp-tilde-phi} if and only if $(\tilde{\mu}_i -1 +\tau)/\tau$ is an eigenvalue of $ \Acalinv{u} ( \hspace{1pt}\lambda v_i -  2 \beta \, \re(u \overline{v_i}) u \hspace{1pt} ) \vert_{\5} $. However, the latter is just the operator in the weighted eigenvalue problem \eqref{weighted-evp1} from Lemma \ref{lemma-weighted-eigenvalueproblem} with eigenvalues $\mu_i$. Hence, we have $\tilde{\mu}_i = (1-\tau)+\tau \mu_i$, where, again due to Lemma \ref{lemma-weighted-eigenvalueproblem}, we have $-1 < \tfrac{-\lambda_1}{\lambda_1+\delta_1} \leq  \mu_i \le \frac{\lambda_1}{\lambda_2} < 1$. We can now distinguish two cases.\\[0.3em]
{\it Case 1:} If $\mu_1 > 0$, then $|\mu_i| \le \frac{\lambda_1}{\lambda_2}$ for all $i\in \mathbb{N}$ and hence
\begin{align*}
1- (1+ \tfrac{\lambda_1}{\lambda_2}) \tau  \le (1-\tau) + \mu_i \tau  \le 1+ (\tfrac{\lambda_1}{\lambda_2} -1 )\tau .
\end{align*}
Here we have $-1 < 1+ (\tfrac{\lambda_1}{\lambda_2} -1 )\tau < 1$ if $0 < \tau < 2 /( 1 -\tfrac{\lambda_1}{\lambda_2} )$ and we have $-1 < 1- (1+ \tfrac{\lambda_1}{\lambda_2}) \tau < 1$ if $\tau < 2/(1+ \tfrac{\lambda_1}{\lambda_2})= 2\lambda_2 /(\lambda_1+\lambda_2)$. Since the second condition is stronger than the first one, we conclude $|(1-\tau) + \mu_i \tau | <1$ for all $\tau \in (0,2\lambda_2 /(\lambda_1+\lambda_2))$.\\[0.3em]
{\it Case 2:} If $\mu_1 < 0$, then we have
\begin{align}
\label{proof_case_2}
(1- \tau) + \tau \mu_1  \le (1-\tau) + \mu_i \tau  \le 1+ (\tfrac{\lambda_1}{\lambda_2} -1 )\tau < 1 \qquad \mbox{for all }
\tau>0,
\end{align}
 where we recall  $-1 < 1+ (\tfrac{\lambda_1}{\lambda_2} -1 )\tau < 1$ for $0 < \tau < 2 /( 1 -\tfrac{\lambda_1}{\lambda_2} )$ from Case 1. As the left hand side of \eqref{proof_case_2} is always strictly smaller than $1$ (for $\tau>0$), we only need to check for which $\tau$ we have $(1- \tau) + \tau \mu_1>-1$. Here we find that this is fulfilled for all $0< \tau  < 2/(1-\mu_1)$. Since $ 2 /( 1 -\tfrac{\lambda_1}{\lambda_2} )>2$ and $ 2/(1-\mu_1)<2$, only the latter condition is relevant. A bound for the latter condition is obtained from $- \mu_1 \le \tfrac{\lambda_1}{\lambda_1+\delta_1}$, which implies $1 + \tfrac{\delta_1}{2\lambda_1+\delta_1}  \le \tfrac{2}{1 - \mu_1}$ and therefore convergence for all $\tau < 1 + \tfrac{\delta_1}{2\lambda_1+\delta_1}$.
\end{proof}
We are now ready to collect the results to prove  Theorem \ref{linear_rate_theorem}.
\begin{proof}[Proof of Theorem \ref{linear_rate_theorem}]
By combining Lemma \ref{Ostrowski-result} and Lemma \ref{lemma-estimates-rho-ast}, we obtain 
\begin{align*}
\| u - \tilde{u}^{n} \|_{H^1(\D)} \leq C_{\eps} \, | \rho^* + \eps |^n \, \| u - u^0 \|_{H^1(\D)} ,
\end{align*}
where $\tilde{u}^n$ denote the auxiliary iterations \eqref{auxi-scheme} with $\tilde{u}^0=u^0$ from a sufficiently small neighborhood of $u$ and $\tau$ as in \eqref{tau-conditions}. Now we can apply Lemma \ref{lemma-relation-un-tilde-un} (which is admissible since $\theta_{u}(u^n) \not= 0$ close to $u$) to obtain the relation $\Theta_{u}(u^{n-1}) u^{n} = \tilde{u}^{n}$ for the original iterates. Consequently
\begin{align}
\nonumber\inf_{\phase_n \in [-\pi,\pi)} 
\|  \exp(\ci \phase_n) \, u^{n}- u \|_{\1} &\le \| \Theta_{u}(u^{n-1}) u^{n} - u \|_{\1}  =  \| \tilde{u}^{n} - u \|_{\1} \\
\label{localestimate-un-phase} & \le C_{\eps} \, | \rho^* + \eps |^n \, \| u - u^0 \|_{H^1(\D)}.
\end{align}
The estimate for the densities follows by Sobolev embedding $L^4(\D) \subset H^1(\D)$ ($d\le3$) as
\begin{eqnarray*}
\lefteqn{ \|  |u^{n}|^2 - |u|^2 \|_{L^2(\D)}^2 \,\,=\,\,  \|  \,|\Theta_{u}(u^{n-1}) u^{n}|^2 - |u|^2 \,\|_{L^2(\D)}^2 } \\
&=&  \| (|\Theta_{u}(u^{n-1}) u^n| -|u|)^2 (|u^n| + |u|)^2 \|_{L^1(\D)} \,\,\le\,\,  \| \tilde{u}^n- u \|_{L^4(\D)}^2  \| |u^n| +|u| \|_{L^4(\D)}^2 \\
&\le& C  \| \tilde{u}^n- u \|_{\1}^2 (  \| u^n \|_{H^1(\D)}^2 + \| u \|_{H^1(\D)}^2 ),
\end{eqnarray*}
where $C>0$ depends on the embedding constant and $\| u^n \|_{H^1(\D)}$ is uniformly bounded depending on the initial energy level. The estimate for the densities is finished with \eqref{localestimate-un-phase}.
\end{proof}
%
%
%
%%%%%%%%%%%%%%%%%%%%%%%

\section{Numerical experiments}
\label{section-numerical-exp}
We conclude with a numerical validation of our theoretical findings. As test setting for this section we consider the energy minimization problem \eqref{minimization-problem} on $\D =[-6,6]^2$, with the harmonic trapping potential 
$V(x,y)=\tfrac{1}{2} \big( (0.9 \, x)^2 +(1.2 \, y)^2 \big)$, the repulsion parameter $\beta=100$ and angular velocity $\Omega=1.2$. With these choices, assumptions \ref{A1}-\ref{A4} are fulfilled. The problem is spatially discretized using $\mathbb{P}^1$-Lagrange finite elements on a uniform mesh with $(2^8 -1)^2$ degrees of freedom. To calculate a sufficiently accurate reference value in this setting, we use a Riemannian conjugate gradient method \cite{AHYY24} with which we obtain the energy of the ground state as $E(u) \approx 1.64547132 $ and the corresponding ground state eigenvalue as $\lm \approx 4.451867515$. Note that both values can still change on finer meshes. The corresponding residual $\|-\Delta u + Vu - \Omega \mathcal{L}_3 u + \beta |u|^2u -\lm u\|_{L^{\infty}(\D)}$ is $9.26975 \times 10^{-10}$ (representing the first order condition $E^{\prime}(u) - \lambda \mathcal{I} u=0$). The computed ground state density $|u|^2$ is depicted in Figure $\ref{plot_density}$ (left).

Before studying the iterative scheme, let us confirm a claim that we made earlier in the paper, which is, that the ground state eigenvalue $\lambda$ does no longer appear at the bottom of the spectrum of the $u$-linearized GP operator $\mathcal{A}_{|u|}$ (as it would be the case for $\Omega=0$ and which was analytically exploited in previous convergence proofs, cf. \cite{CLLZ24,PH24,HeP20,Zhang2022}). Indeed, the spectrum of $\mathcal{A}_{|u|}$ is depicted in Table \ref{Au_spectrum} and we find $\lambda$ at the 17th position (ordered in magnitude) with a large gap to the smallest eigenvalue of $\mathcal{A}_{|u|}$. On the contrary, we also computed the first eigenvalues of $E^{\prime\prime}(u)\vert_{\tangentspace{u}}$. As predicted by Lemma \ref{coercive_lemma}, the ground state eigenvalue $\lambda=\lambda_1\approx 4.451867515$ appears at the bottom of the spectrum with eigenfunction $\ci u$. Furthermore, there is a spectral gap after the first eigenvalue and we computed the second smallest eigenvalue with $\lambda_2=4.463959517$, cf. Table \ref{Esecu_spectrum}.

\begin{table}[h!]
    \centering
    \begin{tabular}{|c|c|c|c|c|c|c|}
        \hline
        $i$ & 1 &  4 &  16 & \cellcolor[gray]{.8} 17 & 23  \\ \hline
        $\tilde{\lambda}_i$ & 3.8044144594   & 3.908068058   &   4.4401477007  & \cellcolor[gray]{.8} 4.451867515  & 4.782714784  \\ \hline
    \end{tabular}
    \caption{Selected eigenvalues $\tilde{\lambda}_i$ of $\mathcal{A}_{|u|}$ (ordered in ascending order based on magnitude). The ground state eigenvalue $\lambda$ appears in the spectrum as 17th eigenvalue (gray) of $\mathcal{A}_{|u|}$.}
    \label{Au_spectrum}
\end{table}

\begin{table}[h!]
    \centering
    \begin{tabular}{|c|c|c|c|c|c|c|}
        \hline
        $i$ & \cellcolor[gray]{.8}1 &  2 &  3 &  4 & 5  \\ \hline
        $\lambda_i$ & \cellcolor[gray]{.8}4.451867515   & 4.463959517   &   4.475871108  & 4.485968615  & 4.503955318  \\ \hline
    \end{tabular}
    \caption{The first five eigenvalues $\lambda_i$ of $E^{\prime\prime}(u)_{\vert \tangentspace{u}}$ (ordered in ascending order based on magnitude). The ground state eigenvalue $\lambda$ appears at the bottom of the spectrum (gray).}
    \label{Esecu_spectrum}
\end{table}
\begin{figure}[h!]
  \centering
  \begin{subfigure}[h]{0.5\textwidth}
    \centering
   \includegraphics[width=0.98\textwidth]{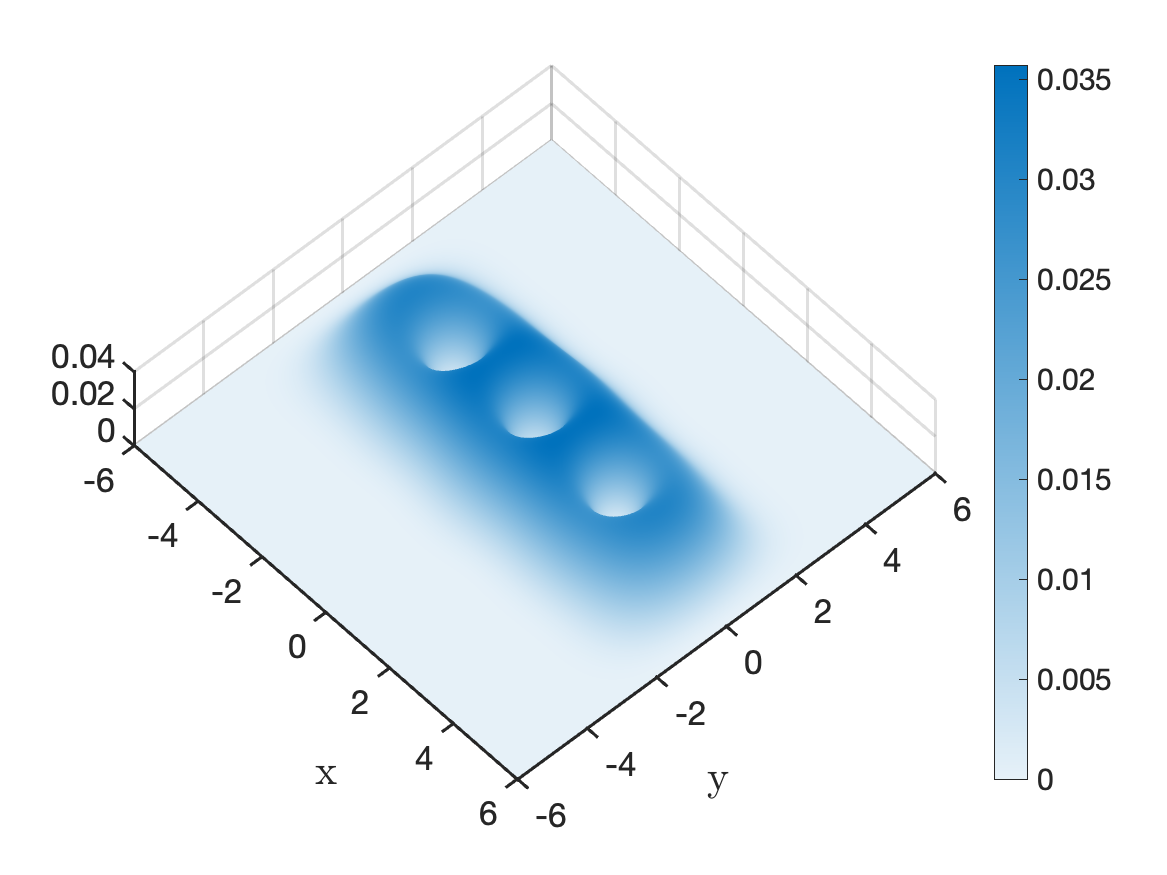}
  \end{subfigure}
  \hfill
  \begin{subfigure}[h]{0.45\textwidth}
     \centering
     \definecolor{c1}{rgb}{0.2,0.5,1.0}
      \definecolor{color4}{rgb}{0.580392156862745,0.403921568627451,0.741176470588235}
     \begin{tikzpicture}
  \begin{axis}[
       legend style={fill=gray!20, font=\small,
        nodes={scale=0.6, transform shape}},
        height = 0.9\textwidth,
        width = 0.9\textwidth,
        xmax   = 2300,  
        xmin   = 600,
        ymax = 10^(-3), 
        ymin = 10^(-9),
        xtick={1000,1500,2000},
        scaled ticks=false,
        yticklabel style={/pgf/number format/sci, /pgf/number format/precision=0}, % Use scientific notation
        ytick={10^(-9),10^(-8),10^(-7),10^(-6),10^(-5),10^(-4)}, % Define y-ticks at specific powers of 10
         xlabel=iteration $n$,
         ylabel=energy error $E(u^n)-E(u)$,    
         xlabel style={font=\small}, 
         ylabel style={font=\small},
         ymode=log
  ]
    \addplot [thick, color=blue] table {adp_vector_new.txt};
    \addplot [thick, color=cyan] table {tau1_vector_new.txt};
    \legend{
adaptive $\tau$,
$\tau=1$,
$E(u)$,
location=northeast,
color=[0 0 0]
}
  \end{axis}
\end{tikzpicture}

  \end{subfigure}
\caption{Left: The ground state density $|u|^2$. Right: Number of iterations versus the energy error $E(u^n)-E(u)$ for the Riemannian gradient method with adaptive $\tau$ and fixed $\tau=1$.
  }
  \label{plot_density}
\end{figure}

%%%%%%%%%%%%%%%%%

\begin{figure}[h!]
  \centering
   \definecolor{c1}{rgb}{0.2,0.5,1.0}
      \definecolor{color4}{rgb}{0.580392156862745,0.403921568627451,0.741176470588235}
      \begin{tikzpicture}
 \begin{axis}[
       legend style={fill=gray!20, font=\small,
        nodes={scale=0.6, transform shape}},
        height = 0.4\textwidth,
        width = 0.9\textwidth,
        xmax   = 8000,  
        xmin   = 0,
        ymax = 10^(-1), 
        ymin = 10^(-9),
        xtick={2000,7000},
        scaled ticks=false,
        yticklabel style={/pgf/number format/sci, /pgf/number format/precision=0}, % Use scientific notation
        ytick={10^(-9),10^(-8),10^(-7),10^(-6),10^(-5),10^(-4),10^(-3),10^(-2)}, % Define y-ticks at specific powers of 10
         xlabel=iteration $n$,
         ylabel=energy error $E(u^n)-E(u)$,    
         xlabel style={font=\small}, 
         ylabel style={font=\small},
         ymode=log
  ]
    \addplot [line width=1pt, color=blue] table {adp_vector.txt};
    \addplot [line width=1pt, color=red] table {M2AN_H1_data.txt};
    \legend{
$a_u$-adaptive metric,
$H^1$-metric,
location=northeast,
color=[0 0 0]
}
  \end{axis}
\end{tikzpicture}

\caption{Number of iterations versus the energy error $E(u^n)-E(u)$ for the Riemannian gradient method with $a_u$-adaptive metric and fixed $H^1$-metric.
  }
  \label{plot_error}
\end{figure}
To initialize our iterative method \eqref{method} we follow \cite[Section 6.1]{BWM05} and select the starting value $u^0$ as the $L^2$-normalized interpolation of the function $u_0(x,y)=\tfrac{(x + \ci y)}{\sqrt{\pi}}  \, \exp({\frac{-(x^2 + y^2)}{2}})$.
We investigate two different choices for $\tau_n$: the adaptively computed optimal step $\tau_n$ for each iteration according to \eqref{tau-n-optimal} and a uniform step size $\tau_n=1$ which corresponds to the inverse iteration $u^{n+1} = \tfrac{ \Acalinv{u^n}u^n }{\| \Acalinv{u^n}u^n \|_{\0}}$. In Figure \ref{plot_density} (right) we observe that the adaptive method performs, as expected, significantly better than the method with fixed $\tau$. For example, if the iterations are stopped if the energy difference $|E(u^n) - E(u)|$ reaches a tolerance of $10^{-9}$, then the (adaptive) Riemannian gradient method requires $1292$ iterations, whereas the inverse iteration ($\tau_n=1$) requires $1958$ iterations to reach a tolerance, which is around 51.5\% more.
Note here that the computational cost per iteration for the adaptive and the non-adaptive method are almost the same when an adequate implementation is used. We refer to Appendix \ref{appendix-B} for details.

Next, we demonstrate the impact of the selected metric by comparing the Riemannian gradient method in the energy-adaptive metric, i.e. \eqref{method}-\eqref{tau-n-optimal}, with the corresponding realization in the standard $H^1$-metric, that is, the general Riemannian (Sobolev) gradient method with the choice $X=H^1_0(\D)$ and the standard inner product $(v,w)_{X} = (v,w)_{L^2(\D)} + (\nabla v , \nabla w)_{L^2(\D)}$. Both realizations use optimal steps $\tau_n$ and require the solving of linear elliptic problems in each iteration. The corresponding results are depicted in Figure \ref{plot_error}, where we observe that the convergence towards the correct limit energy is significantly faster in the energy-adaptive metric. More precisely, the energy-adaptive gradient method requires 1292 iterations to reach an energy error of $10^{-9}$, whereas the $H^1$-gradient method needs 6819 iterations for the same error. It should be noted that both methods show an intermediate energy plateau. This is due to an initial convergence to an excited state, i.e., a critical point of $E$ on $\mathbb{S}$ that is not a ground state. Once the iterates are sufficiently close to the excited state, the iterates \quotes{jump} over the critical point and start to converge to the correct ground state.

Finally, to validate the local convergence results from Theorem \ref{linear_rate_theorem}, we have closer look at the inverse iteration for which we made precise predictions about its local convergence behavior according to estimate \eqref{convergence-inverse-iteration}. Recall here that if $\mu_1>0$ then Theorem \ref{linear_rate_theorem} guarantees linear convergence for $\tau_n=1<  1 +  \frac{\lambda_2 - \lambda_1}{\lambda_2+\lambda_1}$ and it holds (up to an arbitrarily small $\eps$-perturbation and for an $u^0$-dependent constant $C$)
\begin{align}
\label{estimate-for-inverse-iteration}
\inf_{\phase_n \in [-\pi,\pi)} \|  \exp(\ci \phase_n) \, u^{n}- u \|_{\1} \le C \, |\mu_1 |^n \le C \hspace{2pt}\, |\tfrac{\lambda_1}{\lambda_2} |^{n}.
\end{align}
To confirm the estimate, we computed the largest eigenvalues in magnitude $\mu_i$ of the weighted eigenvalue problem (WEVP) given by \eqref{weighted-evp1}, where the five largest ones are shown in Figure \ref{plot_cgs_rate} (right). Together with the previously computed values for $\lambda_1$ and $\lambda_2$, we find that
$$ \mu_{1} \approx 0.99726 \,\, < \,\, \tfrac{\lm_1}{\lm_2}  \approx 0.99729,$$ 
as predicted by Lemma \ref{lemma-weighted-eigenvalueproblem}. The two values are very close to each other with a  difference of merely $| \mu_{1}-\tfrac{\lm_1}{\lm_2}| \approx 3\cdot10^{-5}$. Note that both values are still affected by discretization errors and it is possible that they actually coincide analytically. We finally compare these rates according to \eqref{estimate-for-inverse-iteration} with the numerically observed contraction rates $r(n)$ (for $\tau_n=1$) given by
\begin{align*}
r(n):=  
\frac{\inf\limits_{\phase_{n+1} \in [-\pi,\pi)} \| \exp(\ci \phase_{n+1}) \,  u^{n+1} - u \|_{\1}}{\inf\limits_{\phase_{n} \in [-\pi,\pi)} \| \exp(\ci \phase_{n}) \,  u^{n} - u \|_{\1}}.
\end{align*}
The corresponding results are depicted in Figure \ref{plot_cgs_rate} (left), where we find that $r(n)$ approaches asymptotically a value of around $0.9916$, which is below $|\mu_1|$ and therefore consistent with \eqref{estimate-for-inverse-iteration}. However, we believe that there are initial values $u^0$ for which $|\mu_1|$ and $r(n)$ are asymptotically matching since $|\mu_1|$ can be expected to coincide with the worst-case contraction rate.
In conclusion, the numerical experiments validate the findings of Theorem \ref{linear_rate_theorem}.

\begin{figure}[h]
	\centering
		\begin{subfigure}[h]{0.5\textwidth}
		 \centering
\begin{tikzpicture}
  \begin{axis}[
       legend style={fill=gray!20, font=\large,
        nodes={scale=0.6, transform shape}},
        height = 0.62\textwidth,
        width = 1.1\textwidth,
        xmax   = 1843,  
        xmin   = 1, %
        ymax = 1.05, 
        ymin = 0.95,
        xtick={400,800,1200,1600},
        yticklabel style={/pgf/number format/fixed, /pgf/number format/precision=6},
        ytick={0.9600,1.0000,1.040},
        xlabel=iteration $n$,
        ylabel=contraction rate $r(n)$,    
        xlabel style={font=\small}, 
        ylabel style={font=\small} 
  ]
    \addplot [ thick, color=magenta] table {convergence_rate_latex_file_1.txt};
    \draw[thick,dotted,color=black] (axis cs:0,1) -- (axis cs:1843,1);
     \addplot[thick, color=cyan] coordinates {(1,0.997291194) (2000,0.997291194)};
      \addplot[ color=blue] coordinates {(1,0.997260009) (2000,0.997260009)};
   \legend{
$r(n)$,
$\frac{\lm_1}{\lm_2}= 0.997291194$,
$|\mu_1|=0.997260009$,
location=northeast,
color=[0 0 0]
}
  \end{axis}
\end{tikzpicture}
		\end{subfigure}
		\hfill
		\begin{subfigure}[t]{0.4\textwidth}
		\vspace{-60pt}
	%\centering
	\begin{tabular}{cc}
			\toprule
			%\multicolumn{1}{c}{} & \multicolumn{2}{c}{\textbf{Spectrum of $E''(u)$}}  \\
			%\cmidrule(rl){2} 
			\textbf{$ i$} & {Spectrum of WEVP: $\mu_i$}    \\
			\midrule
			\cellcolor[gray]{.8} 1 & \cellcolor[gray]{.8} 0.997260009  \\
			2 & 0.994321489 \\
			3 & 0.987617747 \\
			4 & 0.941951190 \\
			5 & 0.891180620  \\
%			6 & 0.890196437 \\
%			7 & 0.870511899 \\
			\bottomrule
		\end{tabular}
		\end{subfigure} 
	\caption{Left: Comparison of {\color{magenta}contraction rates} for the inverse iteration (i.e. iteration \eqref{method} with $\tau_n=1$), with {\color{blue}$|\mu_1|$} and the {\color{cyan}first spectral gap} of $E^{\prime\prime}(u)\vert_{\tangentspace{u}}$. The results are consistent with \eqref{estimate-for-inverse-iteration}. Right: Five largest eigenvalues in magnitude of the weighted eigenvalue problem \eqref{weighted-evp1}.  }
	\label{plot_cgs_rate}
\end{figure}
%%%%
%The plot in Figure $\ref{plot_cgs_rate}$  depicts the convergence rate for the basic inverse iteration, where we have used the stopping criterion as $| E(u^{n})-E(u^{n+1}) | \leq 10^{-10}$, so that we are not too close our numerically computed ground state. 

%%%
%%

%%%%%%%%%%%%%%%%%%%%%%---------------------------------------------%%%%%%%%%%%%%%%%%%%%%%

%=============================================================================
%=========  Bibliography
%=============================================================================

{\bf Acknowledgment.} We would like to thank Zixu Feng very much for the insightful comments on the manuscript and in particular the ideas that helped us to significantly improve the lower eigenvalue bound in Lemma \ref{lemma-weighted-eigenvalueproblem}.

\def\cprime{$'$}

\begin{appendix}
\section{Characterization of the spectrum}
\label{appendix-A}

In this appendix, we show that the spectral radius $\rho(\,\tilde{\phi}_{\tau}^{\prime}(u)\,)$ is indeed given by the maximum eigenvalue in magnitude of $\tilde{\phi}_{\tau}^{\prime}(u)$. In particular, we show that the spectrum of $\tilde{\phi}_{\tau}^{\prime}(u)$ can only contain zero and the eigenvalues of  $\tilde{\phi}_{\tau}^{\prime}(u)$. For this, we show that $\tilde{\phi}_{\tau}^{\prime}(u)-(1-\tau)I : H^1_0(\D) \rightarrow H^1_0(\D)$ is a compact operator (where $I$ denotes the identity on $H^1_0(\D)$). Once we have that $\tilde{\phi}_{\tau}^{\prime}(u)-(1-\tau)I$ is compact, we know that the spectrum of $\tilde{\phi}_{\tau}^{\prime}(u)-(1-\tau)I$ only contains $0$ and the eigenvalues of the operator. Since $(1-\tau)$ is just a shift of the spectrum, we conclude that also the spectrum of $\tilde{\phi}_{\tau}^{\prime}(u)$ can only consist of eigenvalues (and zero).

%With the result from Lemma \ref{lemma:tildephi-compact}, the final claim follows now from the lemma below with $T= \tilde{\phi}_{\tau}^{\prime}(u) - (1-\tau)I$, where $I$ denotes the identity. 

Hence, it remains to verify the compactness of $\tilde{\phi}_{\tau}^{\prime}(u)-(1-\tau)I$, which is done in the following lemma.
\begin{lemma}\label{lemma:tildephi-compact}
Let $\tilde{\phi}_{\tau}^{\prime}(u)$ be given by \eqref{tilde-phi-derivative-2} for a ground state $u \in \mathbb{S}$. Then, for any $v \in H^1_0(\D)$ it holds $\tilde{\phi}_{\tau}^{\prime}(u) v - (1-\tau)v \in H^1_0(\D)$ and, for some constant $C$ that depends on $u$, $\lambda$, $\tau$ and the data functions (but not on $v$), it holds 
\begin{align*}
\| \tilde{\phi}_{\tau}^{\prime}(u) v - (1-\tau)v\|_{H^1(\D)} \le C \| v\|_{L^4(\D)}.
\end{align*}
With this, $\tilde{\phi}_{\tau}^{\prime}(u) -(1-\tau)I: H^1_0(\D) \rightarrow H^1_0(\D)$ is a compact operator and its spectrum is given by $0$ and the eigenvalues $\tilde{\mu}_i-(1-\tau)$, where $\tilde{\mu}_i$ is given by \eqref{ev-problem-tilde-phi-prime}.
\end{lemma}
\begin{proof}
%First, we note that our assumptions imply that any ground state admits higher regularity, that is $u \in H^2(\D)$ and hence $u\in L^{\infty}(\D)$ (cf. \cite[Lemma 2.5]{PHMY24}). We shall exploit this in the following calculations without further mentioniung. 
%Next, 
First, recall that
\begin{align*}
\tilde{\phi}_{\tau}^{\prime}(u) v = \psi_{\tau}'(u)\, v - u \, (\psi_{\tau}'(u)v , u)_{\0} + \ci \, u\, \im(\overline{ \int_{\D} \psi_{\tau}'(u)v \, \overline{u} \dx }),
 \end{align*}
 where 
 \begin{align*}
 \psi_{\tau}^{\prime}(u) v  =(1-\tau)v - 2\, \tau \, \langle \mathcal{I} ( \hspace{1pt} v -  \tfrac{\beta}{\lambda} \, \re(u \overline{v}) u )\hspace{1pt} , u \rangle \, u +  \tau \, \lambda \,  \Acalinv{u} \hspace{1pt} (v - 2 \tfrac{\beta}{\lambda} \, \re(u \overline{v}) u \hspace{1pt} ).
\end{align*} 
Since $\|  \Acalinv{u} \hspace{1pt} ( \, \re(u \overline{v}) u \hspace{1pt} ) \|_{H^1(\D)} \le C \| \re(u \overline{v}) u \|_{H^{-1}(\D)} 
%\le C  \| \re(u \overline{v}) u \|_{L^{4/3}(\D)} 
\le C \| u \|_{L^4(\D)}^2 \| v \|_{L^4(\D)}$
for any $v\in H^1_0(\D)$, 
we obtain with some constants $C$ and $\tilde{C}$ (depending on $(u,\lambda,\beta,\Omega,V,\tau)$) that  
\begin{eqnarray*}
\| \tilde{\phi}_{\tau}^{\prime}(u) v - (1-\tau)v\|_{H^1(\D)} &\le& 
\| \psi_{\tau}^{\prime}(u) v - (1-\tau)v \|_{H^1(\D)}  + 2 \| u \|_{H^1(\D)} \| \psi_{\tau}^{\prime}(u) v \|_{L^2(\D)} \\
&\le& \| \psi_{\tau}^{\prime}(u) v - (1-\tau)v \|_{H^1(\D)} + 
C\left(  \| v\|_{L^2(\D)} + \| v\|_{L^4(\D)} \right) \\
&\le& \tilde{C} \| v\|_{L^4(\D)}.
\end{eqnarray*}
Now let $(v_n)_{n \in \mathbb{N}}$ be a bounded sequence in $H^1_0(\D)$, then  the compact embedding of $H^1(\D)$ into $L^4(\D)$ shows that there exists a subsequence (still denoted by $v_n$) that converges weakly in $H^1(\D)$ and strongly in $L^p(\D)$, with $p<6$, to some limit $v^{\ast} \in H^1_0(\D)$. However, the estimate then shows that
\begin{align*}
\| \tilde{\phi}_{\tau}^{\prime}(u) (v_n-v^{\ast}) - (1-\tau)(v_n -v^{\ast}) \|_{H^1(\D)}
\le  \tilde{C} \| v_n -v^{\ast} \|_{L^4(\D)} \rightarrow 0,
\end{align*}
i.e. $\tilde{\phi}_{\tau}^{\prime}(u)v_n- (1-\tau)v_n \rightarrow \tilde{\phi}_{\tau}^{\prime}(u)v^{\ast}- (1-\tau)v^{\ast}$ strongly in $H^1(\D)$.
Hence, any bounded sequence in $H^1_0(\D)$ has a subsequence such that the image of that subsequence under $\tilde{\phi}_{\tau}^{\prime}(u) - (1-\tau)I$ converges strongly in $H^1(\D)$. We conclude that $ \tilde{\phi}_{\tau}^{\prime}(u) - (1-\tau)I : H^1_0(\D) \rightarrow H^1_0(\D)$ is a compact operator. Consequently, the spectrum is given by zero and the eigenvalues of $\tilde{\phi}_{\tau}^{\prime}(u) - (1-\tau)I$. 
\end{proof}

\section{Efficient computation of optimal $\tau$-values}
\label{appendix-B}

We consider the computation of the optimal value $\tau_n$ for the parameter minimization problem \eqref{tau-n-optimal}. For that, assume that a previous approximation $u^n \in H^1_0(\D)$ is given together with a descent direction $d^n= -P_{u^n,X} (\,\nabla_X E(\,u^n\,) \,)\in \tangentspace{u^n}$. We want to find $\tau_n$ such that
\begin{align*}
\tau_n = \underset{\tau \in (0,2)}{\mbox{\normalfont arg\,min}} \, g(\tau),
\qquad \mbox{for } g(\tau) := E \big( \frac{\hspace{-19pt}u^{n} + \tau d^n}{\hspace{3pt}\| u^{n} + \tau d^n \|_{L^2(\D)} } \big).
\end{align*}
The restriction of $\tau$ to the interval $(0,2)$ is motivated by Lemma \ref{energy_reduction_lemma} which shows that the scheme is diverging for $\tau\ge2$. In order to identify the structure of the rational function $g(\tau)$, we express the energy as $E(v) = \tfrac{1}{2} (v,v)_{\sR} +  \tfrac{\beta}{4} \| v \|_{L^4(\D)}^4$ and obtain
\begin{eqnarray*}
 g(\tau) &=& 
 \frac{1}{2} \frac{ ( u^{n} +  \tau \, d^{n} , u^{n} +  \tau \, d^{n} )_{\sR} }{ ( u^{n} +  \tau \, d^{n} , u^{n} +  \tau \, d^{n} )_{L^2(\D)} } + \frac{\beta}{4}  \frac{ \| u^{n} +  \tau \, d^{n} \|_{L^4(\D)}^4 }{ \| u^{n} +  \tau \, d^{n} \|_{L^2(\D)}^4}. 
\end{eqnarray*} 
For the $L^4$-term, we have
\begin{eqnarray*}
\lefteqn{ \| u^{n} +  \tau \, d^{n} \|_{L^4(\D)}^4
\,\,\,=\,\,\, \int_{\D} ( \re [(u^{n} +  \tau \, d^{n} ) \overline{(u^{n} +  \tau \, d^{n} )}] )^2 \dx } \\
%&=&  \int_{\D} ( |u^{n}|^2 + 2 \, \tau\, \re ( u^{n} \overline{d^{n}} ) + \tau^2 |d^{n}|^2  )^2 \dx  \\
&=& \int_{\D}  |u^{n}|^4 + 4 \, \tau\, \re ( u^{n} \overline{d^{n}} ) |u^{n}|^2 +2 \, \tau^2 |d^{n}|^2 |u^{n}|^2 \dx \\
&\enspace&\quad+ \int_{\D}  4 \, \tau^2\, \re ( u^{n} \overline{d^{n}} )^2 + 4 \, \tau^3 |d^{n}|^2 \, \re ( u^{n} \overline{d^{n}} ) + \tau^4 |d^{n}|^4  \dx.
\end{eqnarray*}
By denoting
\begin{align*}
\xi_{0}&:=  \int_{\D}  |u^{n}|^4  \dx, 
\quad
\xi_{1}:= \int_{\D}  \re ( u^{n} \overline{d^{n}} ) |u^{n}|^2 \dx,
\quad
\xi_{2}:= \int_{\D}  |d^{n}|^2 |u^{n}|^2  + 2 [\re ( u^{n} \overline{d^{n}} )]^2 \dx, \\
\xi_{3}&:= \int_{\D} |d^{n}|^2 \, \re ( u^{n} \overline{d^{n}} ) \dx,
\quad
\xi_{4}:= \int_{\D}  |d^{n}|^4  \dx,
\end{align*}
we can write
\begin{align*}
\| u^{n} +  \tau \, d^{n} \|_{L^4(\D)}^4 = \xi_0 + 4 \xi_1 \tau  + 2 \xi_2 \tau^2 + 4 \xi_3 \tau^3 + \xi_4 \tau^4.
\end{align*}
Further introducing
\begin{align*}
\eta_1 := ( u^{n} , d^{n} )_{L^2(\D)} \quad \mbox{and} \quad \eta_2 := ( d^{n} , d^{n} )_{L^2(\D)}
\end{align*}
as well as
\begin{align*}
\zeta_0 :=( u^{n} ,u^{n} )_{\sR} , \qquad \zeta_1 :=( u^{n} ,d^{n} )_{\sR}  \qquad \mbox{and} \quad \zeta_2 :=( d^{n} ,d^{n} )_{\sR} 
\end{align*}
we get
\begin{eqnarray*}
\lefteqn{  g(\tau) \,\,\,=\,\,\, 
 \frac{1}{2} \frac{ ( u^{n} +  \tau \, d^{n} , u^{n} +  \tau \, d^{n} )_{\sR} }{ ( u^{n} +  \tau \, d^{n} , u^{n} +  \tau \, d^{n} )_{L^2(\D)} } + \frac{\beta}{4}  \frac{ \| u^{n}  \,\,+\,\,  \tau \, d^{n} \|_{L^4(\D)}^4 }{ \| u^{n} +  \tau \, d^{n} \|_{L^2(\D)}^4} } \\
 %
%&=&  \frac{1}{2} \frac{ ( u^{n} ,u^{n} )_{\sR}  + 2 \tau \, ( d^{n} , u^{n} )_{\sR} +  \tau^2 ( d^{n}, d^{n} )_{\sR} }{ 1 +
%2 \tau ( d^{n} , u^{n} )_{L^2(\D)} +  \tau^2  ( d^{n} , d^{n} )_{L^2(\D)}  } \\
%&\enspace&\quad +  \frac{\beta}{4}  \frac{ \| u^{n} +  \tau \, d^{n} \|_{L^4(\D)}^4 }{(1 +
%2 \tau ( d^{n} , u^{n} )_{L^2(\D)} +  \tau^2  ( d^{n} , d^{n} )_{L^2(\D)})^2 } \\
%%
%&=&  \tfrac{1}{2} \frac{ \zeta_0 + 2 \tau \, \zeta_1 +  \tau^2 \zeta_2 }{ 1 +2 
%\tau ( d^{n} , u^{n} )_{L^2(\D)} +  \tau^2  ( d^{n} , d^{n} )_{L^2(\D)}  } \\
%&\enspace&\quad +  \frac{\beta}{4}  \frac{ \| u^{n} +  \tau \, d^{n} \|_{L^4(\D)}^4 }{(1 +
%2 \tau ( d^{n} , u^{n} )_{L^2(\D)} +  \tau^2  ( d^{n} , d^{n} )_{L^2(\D)})^2 } \\
%
&=& \frac{  \zeta_0 + 2\zeta_1\,\tau + \zeta_2 \, \tau^2 }{ 2 + 4\, \eta_1 \, \tau + 2 \, \eta_2 \, \tau^2   } 
 \,\,+\,\,  \beta  \frac{ \xi_0 + 4 \xi_1 \tau  + 2 \xi_2 \tau^2 + 4 \xi_3 \tau^3 + \xi_4 \tau^4 }{ 4 + 16 \,\eta_1 \, \tau + (8 \,\eta_2 + 16 \eta_1^2 ) \, \tau^2 + 16 \,\eta_1 \,\eta_2 \, \tau^3 + 4 \,\eta_2^2 \, \tau^4 }.
\end{eqnarray*}
Once the coefficients $\xi_i$, $\eta_i$ and $\zeta_i$ are computed, the minimization of $g: (0,2) \rightarrow \R_{>0}$ can be cheaply realized with a classical method such as Brent's method or golden section search. In all our experiments, the minimizer of $g$ (up to machine precision) is found within milliseconds. 

It remains to discuss the efficient computation of the coefficients. For this let us consider a finite element discretization of the Riemannian gradient method \eqref{method} where the iterates $u^n$ are computed in a finite dimensional space $V_h \subset H^1_0(\D)$ that is spanned by $N$ (nodal) basis functions $\phi_j$. Note that in our interpretation, there are separate basis functions to express the real part and the imaginary part of $u^n$, that is, half of the basis functions are purely real and the other half is purely imaginary (typically of the form $\psi_k$ and $\ci \psi_k$ for a real-valued shape function set $\{ \psi_k \}_{1\le k \le N/2}$). %   and we solve for the real and imaginary parts individually.
For brevity, we still denote the approximations in the discrete setting by $u^n$ to avoid a new notation. 

Let $u^n \in V_h$ be given. We want to determine $u^{n+1}\in V_h$ according to \eqref{method}. For that, we first need to compute $q^n := \Acalinv{u^n}u^n \in V_h$ to obtain the descent direction $d^n=- u^n + \gamma^n q^n$. Recalling the representation \eqref{energy_function-clone}, we see that $q^n := \Acalinv{u^n}u^n \in V_h$ is given as the solution to the (linear) discrete problem
\begin{align}
\label{appendix-qn-problem}
%\int_{\D} \nabla_{\sR} q^n \cdot \overline{\nabla_{\sR}v} + \VR \, q^n \, \overline{v} + \beta |u^n|^2 q^n \, \overline{v} \dx 
(q^n,v)_{\sR} + \beta \, ( |u^n|^2 q^n , v )_{\0} 
=  ( u^n , v )_{\0} 
%\int_{\D} u^n \, \overline{v} \dx 
\qquad \mbox{for all } v\in V_h.
\end{align}
If $\mathbf{q}^n \in \mathbb{R}^{N}$ denotes the corresponding coefficient vector with $q^n = \sum\limits_{j=1}^N \mathbf{q}^n_j \phi_j$ and accordingly $\mathbf{u}^n \in \mathbb{R}^{N}$ for $u^n$, then \eqref{appendix-qn-problem} is equivalent to the linear system seeking $\mathbf{q}^n \in \mathbb{R}^{N}$ with
\begin{align*}
(\mathbf{S} + \mathbf{M}_{u^n}) \hspace{1pt}\mathbf{q}^n = \mathbf{M} \hspace{1pt}\mathbf{u}^n,
\end{align*}
where $\mathbf{S} \in \mathbb{R}^{N \times N}$ is the system matrix and  $\mathbf{M} \in \mathbb{R}^{N \times N}$ the mass matrix with respective entries
\begin{align*}
\mathbf{S}_{kj} = ( \phi_j , \phi_k )_{\sR} 
%\int_{\D} \nabla_{\sR} \phi_j \cdot \overline{\nabla_{\sR}\phi_k} + \VR \, \phi_j \, \overline{\phi_k}  \dx 
\qquad
\mbox{and}
\qquad
\mathbf{M}_{kj} = ( \phi_j , \phi_k )_{\0}  %\int_{\D} \phi_j \, \overline{\phi_k}  \dx
\end{align*}
and where $\mathbf{M}_{u^n} \in \mathbb{R}^{N \times N}$ denotes the nonlinear contribution with entries
\begin{align*}
\mathbf{M}_{u^n}\vert_{kj} = \re \int_{\D} |u^n|^2 \phi_j \, \overline{\phi_k}  \dx
\end{align*}
which has to be updated in each iteration. The update can be costly, depending on the mesh size of the space $V_h$ and the polynomial order of the finite elements. For uniform steps $\tau$, the matrix $\mathbf{M}_{u^n}$ is reassembled after each iteration in a straightforward way through a grid walk. However, for adaptive step sizes $\tau_n$ a different strategy should be used. To see this, let us express the coefficients  $\xi_i$, $\eta_i$ and $\zeta_i$ (needed to construct $g(\tau)$) by matrix vector multiplications. 
First, we directly have 
\begin{align*}
\eta_1 = \langle \mathbf{M} \mathbf{u}^n , \mathbf{d}^n \rangle, \quad
\eta_2 = \langle \mathbf{M} \mathbf{d}^n , \mathbf{d}^n \rangle, \quad
\zeta_0 = \langle \mathbf{S} \mathbf{u}^n , \mathbf{u}^n \rangle, \quad
\zeta_1 = \langle \mathbf{S} \mathbf{u}^n , \mathbf{d}^n \rangle, \quad
\zeta_2 = \langle \mathbf{S} \mathbf{d}^n , \mathbf{d}^n \rangle.
\end{align*}
In a similar way, the factor $\gamma^n =( u^n , q^n )_{\0}^{-1}$ can be expressed through the mass matrix. Since $\mathbf{M}$ and $\mathbf{S}$ are preassembled, the cost are negligible. However, the situation changes when considering the coefficients $\xi_j$. They require the assembly of two new matrices $\boldsymbol{\Xi}_{u^nd^n} \in \mathbb{R}^{N \times N}$ and  $\boldsymbol{\Xi}_{d^nd^n} \in \mathbb{R}^{N \times N}$ with entries
\begin{align*}
\boldsymbol{\Xi}_{u^nd^n} \vert_{kj} = \re \int_{\D} \re ( u^{n} \overline{d^{n}} )\, \phi_j \, \overline{\phi_k}  \dx
\qquad
\mbox{and}
\qquad
\boldsymbol{\Xi}_{d^nd^n} \vert_{kj} = \re \int_{\D} |d^{n}|^2\, \phi_j \, \overline{\phi_k}  \dx.
\end{align*}
Once the matrices are available, the coefficients can be expressed as
\begin{align*}
\xi_0 &= \langle \mathbf{M}_{u^n} \mathbf{u}^n , \mathbf{u}^n \rangle, \quad
\xi_1 = \langle \boldsymbol{\Xi}_{u^nd^n} \mathbf{u}^n , \mathbf{u}^n \rangle, \quad
\xi_2 = \langle \boldsymbol{\Xi}_{d^nd^n} \mathbf{u}^n , \mathbf{u}^n \rangle + 2 \langle \boldsymbol{\Xi}_{u^nd^n} \mathbf{u}^n , \mathbf{d}^n \rangle,\\
\xi_3 &= \langle\boldsymbol{\Xi}_{d^nd^n}\mathbf{u}^n , \mathbf{d}^n \rangle, \quad
\xi_4 = \langle \boldsymbol{\Xi}_{d^nd^n} \mathbf{d}^n , \mathbf{d}^n \rangle.
\end{align*}
The matrices  $\boldsymbol{\Xi}_{u^nd^n}$ and  $\boldsymbol{\Xi}_{d^nd^n}$ can be computed within one single grid walk since they require almost the same evaluations in quadrature points. The computational complexity of computing $\boldsymbol{\Xi}_{u^nd^n}$ and  $\boldsymbol{\Xi}_{d^nd^n}$ is therefore similar to the complexity for assembling $\mathbf{M}_{u^n}$. However, contrary to the case of a uniform step size $\tau$, the adaptive version allows to update $\mathbf{M}_{u^{n+1}}$ from $ \mathbf{M}_{u^n}$, $\boldsymbol{\Xi}_{u^nd^n}$ and  $\boldsymbol{\Xi}_{d^nd^n}$. In fact, for $\tilde{u}^{n+1}:=u^n+ \tau_n d^n$,we obtain
\begin{align*}
 \mathbf{M}_{\tilde{u}^{n+1}} =  \mathbf{M}_{u^n} + 2 \tau_n \, \boldsymbol{\Xi}_{u^nd^n} + \tau_n^2 \, \boldsymbol{\Xi}_{d^nd^n}.
\end{align*}
Consequently, with $u^{n+1} = \tilde{u}^{n+1}/\| \tilde{u}^{n+1} \|_{L^2(\D)}$, we obtain 
$$
 \mathbf{M}_{u^{n+1}} \,\,=\,\, \| \tilde{u}^{n+1} \|_{L^2(\D)}^{-2} \, \mathbf{M}_{\tilde{u}^{n+1}} \,\,=\,\, \langle \mathbf{M} \tilde{\mathbf{u}}^{n+1} , \tilde{\mathbf{u}}^{n+1} \rangle^{-1} \, \mathbf{M}_{\tilde{u}^{n+1}}.
$$
In summary, we find that the gradient method with uniform step size $\tau$ requires the repeated assembly of the matrix $\mathbf{M}_{u^{n}}$, whereas the adaptive version requires the repeated assembly of the matrices $\boldsymbol{\Xi}_{u^nd^n}$ and $\boldsymbol{\Xi}_{d^nd^n}$. The assembly of $\boldsymbol{\Xi}_{u^nd^n}$ and $\boldsymbol{\Xi}_{d^nd^n}$ can be done simultaneously, due to their very similar structure and almost identical function evaluations. Consequently, the computational cost for one iteration with the adaptive and the non-adaptive version of the gradient method \eqref{method} are almost the same, which is why the adaptive version is preferable due to the lower iteration numbers.

\end{appendix}

\end{document}